\documentclass[11pt,a4paper]{amsart}

\usepackage[a4paper]{geometry}
\usepackage[utf8]{inputenc}
\usepackage[T1]{fontenc}
\usepackage{lmodern}
\usepackage{amsmath}
\usepackage{amsfonts,amssymb,amsthm}
\usepackage{mathtools}
\usepackage[all]{xy}
\usepackage{hyperref}
\usepackage{verbatimbox}
\usepackage{tikz-cd}
\usepackage{geometry}
\usepackage{blkarray}
\usepackage{comment}
\usepackage{enumerate}
\usepackage{float}
\usepackage{mathrsfs}
\usepackage{relsize}
\usepackage{tabularx}
\usepackage{pdfpages}

\newtheorem{proposition}{Proposition}
\newtheorem{lemma}{Lemma}

\newtheorem{cor}{Corollary}

\theoremstyle{definition}

\newtheorem{rmk}{Remark}
\newtheorem{convention}{Convention}

\DeclareMathOperator{\ch}{ch}

\DeclareMathOperator{\stab}{Stab}
\DeclareMathOperator{\proj}{Proj}

\DeclareMathOperator{\diag}{diag}

\DeclareMathOperator{\SL}{SL}

\DeclareMathOperator{\coker}{coker}
\DeclareMathOperator{\GL}{GL}
\DeclareMathOperator{\Tor}{Tor}
\DeclareMathOperator{\tors}{tors}
\DeclareMathOperator{\free}{free}

\let\bar\overline

\let\hat\widehat

\let\emptyset\varnothing
\title{On the cohomology of line bundles over certain flag schemes}

\author{Linyuan Liu}

\address{Institut de Mathématiques de Jussieu-Paris Rive Gauche\\
Sorbonne Université -- Campus Pierre et Marie Curie\\
4, place Jussieu -- Boîte Courrier 247\\
F-75252 Paris Cedex 05\\
France}

\email{linyuan.liu@imj-prg.fr}



\date{August 22, 2019} 

\begin{document}
\maketitle

\begin{abstract}
  Let \( G \) be the group scheme \( \SL_{d+1} \) over \( \mathbb{Z}
  \) and let \( Q \) be the parabolic subgroup scheme corresponding to
  the simple roots \( \alpha_{2},\cdots,\alpha_{d-1} \). Then \( G/Q
  \) is the
  \( \mathbb{Z} \)-scheme of partial flags \( \{D_{1}\subset
  H_{d}\subset V\} \). We will calculate the cohomology modules of line
  bundles over this flag scheme. We will prove that the only
  non-trivial ones are isomorphic to the kernel or the cokernel of
  certain matrices with multinomial coefficients.
\end{abstract}
\section*{Introduction}
Fix an integer \( d\geq 2 \). Let \( S=\mathbb{Z}[X_{0},\cdots,X_{d}]
\) be the ring of polynomials over \( \mathbb{Z} \) in the variables
\( X_{0},\cdots,X_{d} \) and for each \( m\in \mathbb{N} \), let \(
S_{m} \) be its graded component of degree \( m \).
Let \( A=\mathbb{Z}[Y_{0},\cdots,Y_{d}] \) be the ring of polynomials
over \( \mathbb{Z} \) in another set of variables \(
Y_{0},\cdots,Y_{d} \) and denote by \( \Delta \) the \( A \)-module of
``inverse'' polynomials:
\begin{equation}
  \label{eq:f82a543df5f928a2}
  \left. \Delta=\mathbb{Z}[Y_{0},\cdots,Y_{d}]_{(Y_{0}\cdots Y_{d})}\middle/
  \sum_{i=0}^{d}\mathbb{Z}[Y_{0},\cdots,Y_{d}]_{(Y_{0}\cdots\hat{Y}_{i}\cdots
    Y_{d})}.\right.
\end{equation}
For each \( n\in\mathbb{N} \), let \( \Delta_{n} \) denote the graded
component of \( \Delta \) of degree \( -n\).
We can easily see that as a \( \mathbb{Z} \)-module, \( \Delta_{n} \) is isomorphic to
\begin{equation}
{\big(\mathbb{Z}[Y_{0}^{-1},\cdots,Y_{d}^{-1}]Y_{0}^{-1}\cdots
Y_{d}^{-1}\big)}_{\deg -n}.\label{eq:f9b7b6cca049f080}
\end{equation}
 Consider
the \( \mathbb{Z} \)-linear map
\begin{equation}\label{eq:phimn}
  \phi=\phi_{m,n}: S_{m-1}\otimes \Delta_{n+d+1}\to  S_{m}\otimes \Delta_{n+d}
\end{equation}
given by the multiplication by the element \(
f=X_{0}\otimes Y_{0}+\cdots+X_{d}\otimes Y_{d} \). The goal (partially achieved) is to
study the cokernel of \( \phi \). Furthermore, there is a natural
action of the group scheme \( G=\SL_{d+1} \) on the representation \(
V \) with basis \( X_{0},\cdots,X_{d} \) and on the dual representation \(
V^{*} \) with dual basis \( Y_{0},\cdots,Y_{d} \), and the element \( f
\) is \( G \)-invariant, hence \( \coker(\phi) \) and \( \ker(\phi) \)
are \( G \)-modules. As will be explained below, these are the cohomology
groups (the only non zero ones) of a certain line bundle \( \mathcal{L}=\mathcal{L}(m,-n-d) \)
on the \( \mathbb{Z} \)-scheme of partial flags \( D_{1}\subset
H_{d}\subset V \).

\section{Notations}
Let \( G \) be the group scheme \( \SL_{d+1} \) over \( \mathbb{Z} \) with \( d\geq 2
\). Let  \( T \) and \( B \) be the subgroup schemes of diagonal
matrices and of  {\it lower} triangular matrices respectively. Let \(
W \) be the Weyl group of \( (G,T) \) and  \( X(T) \)
 the character group of \( T \). For \(
i\in\{0,1,\cdots,d\} \), we define \( \epsilon_i\in X(T) \)  as the
character that sends \( \diag(a_0,a_1,\cdots,a_d) \) to \( a_{i} \)
and we set  $\alpha_i=\epsilon_{i-1}-\epsilon_{i}$. Then \(
\{\alpha_{1},\cdots,\alpha_{d}\} \) is the set of simple roots. We
denote by
$\omega_1,\cdots,\omega_d$ the corresponding fundamental weights and
by \( R^{+} \) the set of positive roots. Let
\( X(T)^{+}\subset X(T) \) be the set of dominant weights and let \(
\rho\in X(T) \) be the half sum of positive roots. The dot action of
the Weyl group is defined by \( w\cdot \lambda=w(\lambda+\rho)-\rho
\), for all \( w\in W \) and \( \lambda\in X(T) \). Let \(
C=\{\lambda\in X(T)\mid \lambda+\rho\in X(T)^{+}\}\). 

If \( N \) is a \( B \)-module, we set
\(H^{i}(N)=H^{i}(G/B,\mathcal{L}(N))  \) where \( \mathcal{L}(N) \) is
the \( G \)-equivariant vector bundle on the flag scheme \( G/B \)
induced by \( N \) (cf. \cite{Jan03} I.5.8). In
particular, if \( \mu\in X(T) \), then \( \mu \) can be viewed as a
one-dimensional \( B \)-module, and we set \( H^{i}(\mu)=H^{i}(G/B,\mathcal{L}(\mu)) \).

We fix \(m,
n\in\mathbb{N} \) and take
$\mu=m\omega_1-(n+d)\omega_d$. Our goal is to calculate the cohomology
groups \( H^{i}(\mu) \) of the line bundle \( \mathcal{L}(\mu) \). The only non zero ones are \( H^{d-1}(\mu)\cong
\ker(\phi_{m,n}) \) and \( H^{d}(\mu)\cong\coker(\phi_{m,n}) \) and we
will show that \( H^{d}(\mu) \) is isomorphic to the cokernel of a
certain matrix of multinomial coefficients of size much smaller than
the rank of the \( \mathbb{Z} \)-modules \( S_{m-1}\otimes \Delta_{n+d+1} \)
and \( S_{m}\otimes \Delta_{n+d} \).

\section{Description of the cohomology groups \( H^{d}(G/P,\mu) \)}

Let \( V \) be the natural representation of  \( G \) and  \( V^{*} \)
the dual representation. Let \( \{X_{0},\allowbreak X_{1},\allowbreak
\cdots, \allowbreak X_{d}\} \) be the
canonical basis of \( V \) and let \( \{Y_{0},Y_{1},\cdots,Y_{d}\} \)
be the dual basis of 
 \( V^{*} \). Let \( P_{d} \) and \( P_{1} \) be the stabilizers of
 the point \( [X_{d}] \in \mathbb{P}(V)\) and of the point \(
 [Y_{0}]\in\mathbb{P}(V^{*}) \) respectively. Let \( Q=P_{d}\cap P_{1}
 \). Then \(P_{d}  \) (resp. \( P_{1} \), resp. \( Q\)) is the
 parabolic subgroup scheme containing \( B \) and corresponding to the
 simple roots \( \alpha_{1},\alpha_{2},\cdots,\alpha_{d-1} \) (resp. \(
 \alpha_{2},\alpha_{3},\cdots,\alpha_{d} \), resp. \(
 \alpha_{2},\cdots,\alpha_{d-1} \)). 
Therefore, denoting by \( S(V) \) resp. \( S(V^{*}) \) the symmetric
algebra of \( V \) resp. \( V^{*} \) one has
\begin{eqnarray*}
  G/P_{d}\cong& \mathbb{P}(V) =&\proj(S(V^*))=\proj(k[Y_{0},Y_{1},\cdots,Y_{d}])\\
  G/P_{1} \cong & \mathbb{P}(V^*)=&\proj(S(V))=\proj(k[X_{0},X_{1},\cdots,X_{d}]).
\end{eqnarray*}

We have for all \( r\in\mathbb{Z} \) (cf. \cite{Jan03} II.4.3)
\begin{equation}
  \label{eq:3305974d3532dde5}
  \mathcal{L}_{G/P_{1}}(r\omega_{1})\cong \mathcal{O}_{\mathbb{P}(V^{*})}(r),
\end{equation}
hence
\begin{equation}
  \label{eq:suite6*}
H^{0}(G/P_{1},r\omega_{1})\cong S_{r}  
\end{equation}
if \( r\geq 0 \), where \(S_{r}=\langle
X_{0}^{a_{0}}X_{1}^{a_{1}}\cdots
X_{d}^{a_{d}}|a_{0}+a_{1}+\cdots+a_{d}=r\rangle\) as in the introduction..

On the other hand, for \( P_{d} \), we have for all \( r\in\mathbb{Z} \) 
\begin{equation}
  \label{eq:suite2*}
  \mathcal{L}_{G/P_{d}}(r\omega_{d})\cong\mathcal{O}_{\mathbb{P}(V)}(r).
\end{equation}
 Hence if 
 \( r\geq 0 \) we have (cf. \cite{Kem93} Cor 9.1.2):
\begin{equation}
  \label{eq:suite5*}
 H^{d}(G/P_{d},-r\omega_{d})\cong \Delta_{r}
\end{equation}
where \(\Delta_{r}=\langle Y_{0}^{-1-b_{0}}Y_{1}^{-1-b_{1}}\cdots
Y_{d}^{-1-b_{d}}| b_{i}\in
\mathbb{N},b_{0}+b_{1}+\cdots+b_{d}+d+1=r\rangle\) as in \eqref{eq:f9b7b6cca049f080}.

We set \(\xi=([Y_{0}], [X_{d}])\in \mathbb{P}(V^{*})\times
\mathbb{P}(V)\). Then
\begin{displaymath}
  Q=\stab(\xi)\text{ and }G/Q\cong G\xi=\mathscr{V}(X_{0}Y_{0}+X_{1}Y_{1}+\cdots+X_{d}Y_{d})
\end{displaymath}
where \( \mathscr{V}(\psi) \) is the closed subscheme  defined by
 a bi-homogeneous polynomial  \( \psi \). This means that \( G/Q \) is
 the flag
 scheme \(\{D_{1}\subset H_{d}\subset V\} \), which is a hypersurface in
 \( \mathbb{P}(V^{*})\times \mathbb{P}(V) \).

Denote \( \mathbb{P}(V^{*})\times \mathbb{P}(V)\) by \( Z \). Then \(
\mathcal{O}_{Z}\cong
\mathcal{O}_{\mathbb{P}(V^{*})}\boxtimes\mathcal{O}_{\mathbb{P}(V)} \)
by Künneth formula. The ideal sheaf defining the subvariety \(
G/Q=\mathscr{V}(f) \) is \( \mathcal{L}(-1,-1) \). More precisely, we
have an exact sequence of sheaves
\begin{equation}
  \label{eq:957ac4756912efef}
  0\to\mathcal{L}(-1,-1)\xrightarrow{f}\mathcal{O}_{Z}\to\mathcal{O}_{G/Q}\to 0,
\end{equation}
i.e.
\begin{equation}
  \label{eq:24b84bcd2c34f7a1}
  0\to\mathcal{L}_{G/P_{1}}(-1)\boxtimes
  \mathcal{L}_{G/P_{d}}(-1)\xrightarrow{f}\mathcal{O}_{\mathbb{P}(V^{*})}\boxtimes
  \mathcal{O}_{\mathbb{P}(V)}\to \mathcal{O}_{G/Q}\to 0,
\end{equation}
where \( f \) means the multiplication by the element \(
f=X_{0}\otimes Y_{0}+\cdots+X_{d}\otimes Y_{d}\).

Hence for all \( m,n\in\mathbb{N} \), by tensoring
\eqref{eq:24b84bcd2c34f7a1} with \(
\mathcal{L}_{G/P_{1}}(m\omega_{1})\boxtimes\mathcal{L}_{G/P_{d}}(-(n+d)\omega_{d})
\), we obtain an exact sequence:
\begin{equation}
  \label{eq:suite1*}
 0 \to \mathcal{O}_{G/P_{1}}(m-1)\boxtimes
    \mathcal{O}_{G/P_{d}}(-n-d-1)
    \xrightarrow{f}
    \mathcal{O}_{G/P_{1}}(m)\boxtimes
    \mathcal{O}_{G/P_{d}}(-n-d)\to \mathcal{L}_{G/Q}(\mu) \to 0. 
\end{equation}
By taking cohomology, we obtain \( H^{i}(G/Q, \mu)=0 \) if \( i\neq
d-1,d \) and 
 an exact sequence of
\(G\)-modules:
\begin{equation}
    0\to H^{d-1}(G/Q, \mu)\to  S_{m-1}\otimes \Delta_{n+d+1} \xrightarrow{f}
    S_{m}\otimes \Delta_{n+d}\to H^{d}(G/Q, \mu)\to 0.
\label{eq:8d3f39d1d8e8a856}
\end{equation}
Since \( H^{0}(Q/B,\mu)\cong \mu \) and \( H^{i}(Q/B,\mu)=0 \) if \(
i>0 \), we have
\begin{displaymath}
  H^{i}(\mu)\cong H^{i}(G/Q,\mu)
\end{displaymath}
for all \( i \). So \eqref{eq:8d3f39d1d8e8a856} gives that \(
H^{d-1}(\mu)=\ker(f) \) and \( H^{d}(\mu)=\coker(f) \).

Let \( \sigma_{1},\cdots,\sigma_{d} \) be the simple reflections, then since \(
\mu=(m,0,\cdots,0,-n-d) \), we have \( \sigma_{d}\cdot
\mu=(m,0,\cdots,0,-n-d+1,n+d-2)\), then \( \sigma_{3}\sigma_{4}\cdots \sigma_{d}\cdot
\mu=(m,-n-2,n+1,0,\cdots,0)\) and \( \sigma_{2}\cdots \sigma_{d}\cdot
\mu=(m-n-1,n,0,\cdots,0) \). Hence \( \mu\in \sigma_{d}\cdots \sigma_{2}\cdot C
\) if \( m\geq n \) and \( \mu\in \sigma_{d}\cdots \sigma_{1}\cdot C \) if \(
n>m \). In particular, \( \mu \) is regular unless \( m=n \), and if
\( m=n \), \( \mu \) is located on a unique wall.

For a field \( k \) and any \( i \), set \(
H_{k}^{i}(\mu)=H^{i}(G_{k}/B_{k},\mu) \), where \( G_{k}
\) and \( B_{k} \) are the \( k \)-group schemes obtained by base
change. Then  we have an exact sequence
\begin{equation}
  \label{eq:4550b9f5b03b97a6}
  0\to H^{i}(\mu)\otimes k\to H^{i}_{k}(\mu)\to
  \Tor_{1}^{\mathbb{Z}}(k,H^{i+1}(\mu))\to 0
\end{equation}
by the universal coefficient theorem (cf. \cite{Jan03} I.4.18). Since
\( H^{d-1}(\mu) \) is a free \( \mathbb{Z} \)-module by
\eqref{eq:8d3f39d1d8e8a856}, it is completely
determined by \( H^{d-1}(\mu)\otimes \mathbb{Q} \). On the other hand, since
the extension \( \mathbb{Z}\to\mathbb{Q} \) is flat, we have \(
H^{d-1}(\mu)\otimes \mathbb{Q}\cong H_{\mathbb{Q}}^{d-1}(\mu)
\) by \eqref{eq:4550b9f5b03b97a6}, and the latter can be calculated by
the Borel-Weil-Bott theorem (cf. \cite{Jan03} II.5.5). More precisely,
we have \( H^{d-1}_{\mathbb{Q}}(\mu)=0 \) if \( n> m \), and \( \ch
H^{d-1}_{\mathbb{Q}}(\mu)=\chi (m-n-1,n,0,\cdots,0)\) if \( m\geq n
\), where \( \ch M \) is the character of \( M \) (cf. \cite{Jan03} I
2.11 (6)), and \( \chi(\mu) \) is the Euler characteristic of \( \mu \)
viewed as a \( B \)-module (cf. \cite{Jan03} II.5.7), which can be
calculated by the Weyl's character formula (cf. \cite{Jan03}
II.5.10). So the most interesting group is \( H^{d}(\mu)\cong
\coker(f) \), which can have torsion. We have an exact sequence of \(
\mathbb{Z} \)-modules
\begin{equation}
  \label{eq:5d56c3e640125b6e}
  0\to H^{d}(\mu)_{\tors}\to H^{d}(\mu)\to H^{d}(\mu)_{\free}\to 0.
\end{equation}
Since \(
H^{d+1}(\mu)=0 \),  for any field \( k \) we have \(
H_{k}^{d}(\mu)\cong H^{d}(\mu)\otimes k \) by
\eqref{eq:4550b9f5b03b97a6}. Tensoring \eqref{eq:5d56c3e640125b6e} by
\( k \) and using the fact that \( H^{d}(\mu)_{\free} \) is torsion free, we thus get
\begin{equation}
  \label{eq:cf5306390199b6af}
   0\to H^{d}(\mu)_{\tors}\otimes k \to H^{d}_{k}(\mu)\to
   H^{d}(\mu)_{\free}\otimes k\to 0 .
 \end{equation}
 First, take \( k=\mathbb{Q} \), this gives an isomorphism \(
  H^{d}(\mu)_{\free}\otimes \mathbb{Q}\cong
  H^{d}_{\mathbb{Q}}(\mu)\), which
 can be calculated by the Borel-Weil-Bott theorem and the Weyl's
 character formula, so \( H^{d}(\mu)_{\free} \) is already
 known. On the other hand, we have
 \begin{equation}
   \label{eq:15d1900947e763cc}
  0\to H^{d-1}(\mu)\otimes k\to H^{d-1}_{k}(\mu)\to
  \Tor_{1}^{\mathbb{Z}}(k,H^{d}(\mu)_{\tors})\to 0. 
\end{equation}
Hence \( H^{d}(\mu)_{\tors} \) determines both \( H^{d-1}_{k}(\mu) \)
and \( H^{d}_{k}(\mu) \) for any field \( k \). Therefore, it suffices
to calculate \( \coker(f)\cong H^{d}(\mu) \) (especially its torsion part) to achieve our goal.

 We set $E= S_{m-1}\otimes \Delta_{n+d+1}$ and  $F= S_{m}\otimes \Delta_{n+d}$. The highest weight of  \( E \) and \( F \) is $(m+n-1)\omega_1$. 

We know that $X_0$, $X_1$,$\cdots$, $X_d$ are of weights $\omega_1$,
$\omega_1-\alpha_1$, $\cdots$,
$\omega_1-\alpha_1-\alpha_2-\cdots-\alpha_d=-\omega_{d}$ and $Y_i$ is of opposite
weight to $X_i$. Since  $f$ preserves the weight spaces, we can
restrict \( f \) to the \( \nu \)-weight space for each dominant
weight \( \nu \), and we get a linear map
\( f_{\nu}: E_{\nu}\to F_{\nu} \), where \( E_{\nu} \) and \( F_{\nu}
\) are the \( \nu \)-weight spaces of \( E \) and \( F \) respectively. Hence it suffices
to calculate the cokernel of  $f_{\nu}$ for each dominant weight
$\nu\leq (m+n-1)\omega_{1}$, where \( \leq \) is the usual partial
order on \( X(T) \).

For each such \( \nu \), there exists \(
s_{1},s_{2}\cdots, s_{d}\in \mathbb{N} \) such that
\begin{align*}
\nu=&(m+n-1)\omega_1-s_1\alpha_1-s_2\alpha_2-\cdots-s_d\alpha_d\\
=&(m+n-1-2s_1+s_2)\omega_1+(s_1-2s_2+s_3)\omega_2+(s_2-2s_3+s_4)\omega_3\\
&+\cdots+(s_{d-2}-2s_{d-1}+s_d)\omega_{d-1}+(s_{d-1}-2s_d)\omega_d
\end{align*}
with $m+n-1-2s_1+s_2\geq 0$ and $s_{i-1}-2s_i+s_{i+1}\geq 0$ for
$2\leq i\leq d-1$ and $s_{d-1}-2s_d\geq 0$. Hence the
$\nu$-weight space consists of monomials \( X_{0}^{a_{0}}\cdots
X_{d}^{a_{d}}Y_{0}^{-1-b_{0}}\cdots Y_{d}^{-1-b_{d}} \) such that $(a_0+b_0, a_1+b_1, \cdots,
a_d+b_d)=(m+n-1-s_1, s_1-s_2, \cdots,s_{d-1}-s_d,s_d)$. Therefore, if
we fix $\nu$, a monomial in the $\nu$-weight space is determined by
$b=(b_0, b_1,\cdots, b_d)$. In the following, the letter $b$ without
subscript means a tuple of non-negative integers.

\section{The case \( n\leq m \)}\label{section:nleqm}

\paragraph{\ref{section:nleqm}.1}
If $s_1\leq n-1$, then a monomial $(b_0,b_1,\cdots,b_d)$ in
$E_{\nu}$ satisfies $b_0\geq 1$ since $b_1+\cdots+b_d\leq s_1<n$ and
$b_0+b_1+\cdots+b_d=n$. Let
\begin{multline*}
A=\{b=(b_0,b_1,\cdots,b_d)\in
\mathbb{N}^{d+1}\mid b_0+b_1+\cdots+b_d=n,\quad 1\leq b_0\leq
m+n-1-s_1,\\\quad b_i\leq s_i-s_{i+1}\quad \text{if} \quad 1\leq i\leq
d-1,\; b_d\leq s_d \}.
\end{multline*}
Then we take as a basis for \( E_{\nu} \) the set
 $\{v_b\mid  b\in A\}$ where $v_b$ is the monomial determined by $b$. Since \(
 m+n-1-s_{1}\geq 2n-1-s_{1}\geq n \), we have
\begin{multline*}
A=\{b=(b_0,b_1,\cdots,b_d)\in
\mathbb{N}^{d+1}\mid b_0+b_1+\cdots+b_d=n,\quad 1\leq b_0\leq
n,\\\quad b_i\leq s_i-s_{i+1}\quad \text{if} \quad 1\leq i\leq
d-1, \; b_d\leq s_d \}.
\end{multline*}
 For
 $0\leq i\leq d$, we set
 $e_i=(0,0,\cdots,1,\cdots,0)\in\mathbb{Z}^{d+1}$ where $1$ is at the
 $i$-th position, then we take as a basis for \( F_{\nu} \) the set   
  $\{w_{b}\mid  b\in A\} $  where $w_b$ is the
  monomial determined by
 $b-e_0$. We set \( w_{b}=0 \) whenever \( b\notin
 A \).  With these notations, we have for all \( u\in A \)
$$
f(v_u)=w_{u}+w_{u+e_0-e_1}+\cdots+w_{u+e_0-e_d}.
$$ 

We equip the set $A\subset\mathbb{N}^{d+1}$
with the reverse lexicographic  order. Then $u+e_0-e_i<u$ for all
$1\leq i\leq d$, thus the matrix of $f$ with respect to the bases $v_b$
and $w_b$ is lower triangular, and its  entries on the diagonal
are all \( 1 \). Hence the cokernel $H^d(\mu)_{\nu}$ of \( f_{\nu} \)
is zero if \( s_{1}<n \).

This proves that every weight of \( H^{d}(\mu) \) is \( \leq (m+n-1)\omega_{1}-n\alpha_{1} \).

\medskip
\paragraph{\ref{section:nleqm}.2}
If \( s_{1}\geq n \), 
set $s_1=n+k$ with $k\geq0$. We introduce
\begin{align*}
 A=&\{b=(b_0,b_1,\cdots,b_d)\in \mathbb{N}^{d+1}\mid b_0+b_1+\cdots+b_d=n,\quad 1\leq b_0\leq m-1-k,\\
&b_i\leq s_i-s_{i+1}\quad \text{if} \quad 1\leq i\leq d-1, \quad b_d\leq s_d \},\\ 
C=&\{b=(0,b_1,\cdots,b_d)\in \mathbb{N}^{d+1}\mid b_1+\cdots+b_d=n,\\
& b_i\leq s_i-s_{i+1}\quad \text{if} \quad 1\leq i\leq d-1, \quad
                                                                        b_d\leq s_d \},\\ D=&\{b=(m-k-1,b_1,\cdots,b_d)\in \mathbb{N}^{d+1}\mid b_1+\cdots+b_d=k+n-m, \quad b_i\leq s_i-s_{i+1}\\
& \text{if} \quad 1\leq i\leq d-1,\quad b_d\leq s_d \}.
\end{align*}

We take the set $\{v_b\mid b\in A\cup C\}$ as the basis of $E_{\nu}$ where
$v_b$ is the monomial determined by $b=(b_0,\cdots,b_d)$. We take the
set $\{w_b\mid b\in A\}\cup \{u_b\mid b\in D\}$ as the basis of $F_{\nu}$
where $w_b$ is the monomial determined by $b-e_0$ and $u_b$ is the monomial determined by $b\in D$. 
\begin{convention}
  Let $b\in \mathbb{Z}^{d+1}$. If $b\notin A\cup C$,  we set
  $v_b=0$. If $b\notin A$, we set $w_b=0$. If $b\notin D$, we set $u_b=0$.
\end{convention}

With these notations, we have
\begin{equation}
  \begin{aligned}
    f(v_b)&=w_b+w_{b+e_0-e_1}+w_{b+e_0-e_2}+\cdots w_{b+e_0-e_d}\quad \text{if} \quad b_0\leq m-k-2\\
    f(v_b)&=w_b+u_{b-e_1}+u_{b-e_2}+\cdots+u_{b-e_d} \quad
    \text{if}\quad b_0=m-k-1.
  \end{aligned}
\end{equation}

Now we make a change of basis of  $E_{\nu}$ by defining
$v'_b$ for all $b\in A\cup C$ by:

\begin{itemize}
\item if $b_0=m-k-1$, we set $v'_b=v_b$;

 \item  if $b_0=m-k-2$, we set
  \begin{equation}
    v'_b=v_b-v_{b+e_0-e_1}-v_{b+e_0-e_2}-\cdots-v_{b+e_0-e_d} .
  \end{equation}
\end{itemize}

Let \( j\geq 1 \). If we have already defined  $v'_b$ for all $b$
such that
$j\leq b_0\leq m-k-2$, we set
\begin{equation}
  v'_b=v_b-v'_{b+e_0-e_1}-v'_{b+e_0-e_2}-\cdots-v'_{b+e_0-e_d} \quad \text{if}\quad b_0=j-1.
\end{equation}
Hence $v'_b$ is defined for all $b\in A\cup C$.

Therefore, if $b_0=m-k-1$, we have:
\begin{equation}
  \begin{aligned}
    f(v'_b)&=w_b+u_{b-e_1}+u_{b-e_2}+\cdots+u_{b-e_d}\\
    &=w_{b}+\binom{1}{1,0,\cdots,0}u_{b-e_{1}}+\binom{1}{0,1,0,\cdots,0}u_{b-e_{2}}+\cdots+\binom{1}{0,\cdots,0,1}u_{b-e_{d}}\\
    &=w_{b}+\sum_{b'\in D}\binom{1}{b_{1}-b_{1}',b_{2}-b_{2}',\cdots,b_{d}-b_{d}'}u_{b'}.
  \end{aligned}
  \label{eq:32ac0045f9a7c897}
\end{equation}

\begin{lemma}
  For all $b\in
A\cup C$, we have 
\begin{equation}
  \label{eq:5555}
  f(v'_b)=w_b-(-1)^{m-k-b_{0}}\sum_{b'\in D}\binom{m-k-b_0}{b_1-b_1',b_2-b_2',\ldots,b_d-b_d'}u_{b'}.
\end{equation}
\end{lemma}

\begin{rmk}
  Since $b_0+\cdots+b_d=n$ if $b\in A\cup C$, we have
  $b_1+b_2+\cdots+b_d=n-b_0$. On the other hand, if $b'\in D$, then $b'_1+b'_2+\cdots+b'_d=n-1-b'_0=n-1-(m-1-k)=k+n-m$, hence $b_1-b_1'+b_2-b_2'+\cdots+b_d-b_d'=m-k-b_0$.
\end{rmk}

\begin{proof}
  We use  descending induction on \(b_{0}\).
  Clearly, \eqref{eq:5555} is true if $b_0=m-k-1$.  If
  \eqref{eq:5555} holds for all $b\in A\cup C$ such that
  $1\leq j\leq b_0\leq m-k-1$, then for all \( b\in A\cup C \) such
  that $b_0=j-1$, one has:
  \begin{align*}
    f(v'_b)=&\; f(v_b)-f(v'_{b+e_0-e_1})-\cdots-f(v'_{b+e_0-e_d})\\
    =&\; w_b+w_{b+e_0-e_1}+w_{b+e_0-e_2}+\cdots+w_{b+e_0-e_d}\\
            &-w_{b+e_0-e_1}+(-1)^{m-k-b_{0}-1}\sum_{b'\in D}\binom{m-k-b_0-1}{b_1-1-b_1', b_2-b_2',\cdots, b_d-b_d'}u_{b'}\\
            &-w_{b+e_0-e_2}+(-1)^{m-k-b_{0}-1}\sum_{b'\in D}\binom{m-k-b_0-1}{b_1-b_1', b_2-1-b_2',\cdots, b_d-b_d'}u_{b'}\\
            &-\cdots\\
           &-w_{b+e_0-e_d}+(-1)^{m-k-b_{0}-1}\sum_{b'\in D}\binom{m-k-b_0-1}{b_1-b_1', b_2-b_2',\cdots, b_d-1-b_d'}u_{b'}\\
    =&\; w_b-(-1)^{m-k-b_{0}}\sum_{b'\in D}\sum_{i=1}^d\binom{m-k-b_0-1}{b_1-b_1',b_2-b_2',\ldots,b_i-1-b_i',\ldots,b_d-b_d'}u_{b'}\\
    =&\; w_b-(-1)^{m-k-b_{0}}\sum_{b'\in D}\binom{m-k-b_0}{b_1-b_1',b_2-b_2',\ldots,b_d-b_d'}u_{b'}.
  \end{align*}

  This proves the lemma.
\end{proof}

Therefore, by writing \( f(v'_{b}) \) in row, the matrix of $f$ with respect to the bases $v'_b$ and $w_b,
u_b$ is of the form
 \begin{equation*}
  \begin{array}{cc}
&  \begin{array}{ccclcccc} & &  & A  & & & &D \end{array} 
\\ 
\begin{array}{c}
\\
\\
A
\\
\\
\\
C
\end{array} 
 & \left(
\begin{array}{cccc|c}
1 & 0 & \cdots & 0 & *
\\
0 & 1 & \ddots & \vdots & * 
\\
\vdots &\ddots & \ddots & 0 & \vdots 
\\
0 & \cdots & 0 & 1 & * 
\\
\hline 
0& \cdots & 0 & 0 
&
M
\end{array}
\right)
\end{array} 
\end{equation*}
where the rows of $M$ are indexed by $C$, its columns  by $D$, and the
entry corresponding to \(  b\in C \) and \( b'\in D \)
is
$(-1)^{m-k+1}\binom{m-k}{b_1-b_1',b_2-b_2',\ldots,b_d-b_d'}$.
(One has \( m-k-b_{0}=m-k \) since \( b_{0}=0 \) for \( b\in C \)).

We thus obtain the following proposition.
\begin{proposition}Let \( m\geq n\geq 0 \)
  \begin{enumerate}
  \item Every weight of \(
    H^{d}(m,0,\cdots,0,-n-d) \) is \( \leq (m-n-1,n,0,\cdots,0) \).
  \item For \( (k,s_{2},\cdots,s_{d}) \) such
    that
    \begin{displaymath}
      \nu=(m-n-1,n,0,\cdots,0)-k\alpha_{1}-s_{2}\alpha_{2}-\cdots-s_{d}\alpha_{d}
    \end{displaymath}
    is dominant, the \( \nu \)-weight space of
    \( H^{d}(m,0,\cdots,0,-n-d) \) is isomorphic as an abelian group
    to the cokernel of the matrix
    \begin{equation}
      \label{eq:7e42ab2aaf16c45f}
      \Bigg(\binom{m-k}{b_1-b_1',b_2-b_2',\ldots,b_d-b_d'}\Bigg)_{\substack{b\in
          C\\b'\in D}}
    \end{equation}
    where by setting \( s_{1}=n+k \), we have
    \begin{align*}
      C=&\{b=(0,b_1,\cdots,b_d)\in \mathbb{N}^{d+1}\mid b_1+\cdots+b_d=n,\\
        & b_i\leq s_i-s_{i+1}\quad \text{if} \quad 1\leq i\leq d-1, \quad
          b_d\leq s_d \},\\
      D=&\{b=(m-k-1,b_1,\cdots,b_d)\in
          \mathbb{N}^{d+1}\mid b_1+\cdots+b_d=k+n-m, \quad b_i\leq s_i-s_{i+1}\\
        & \text{if} \quad 1\leq i\leq d-1,\quad  b_d\leq s_d \}.
    \end{align*}
  \end{enumerate}

\end{proposition}

In this case, we also know that \( H^{d}(\mu) \) is a torsion abelian
group, since \( H^{d}(\mu)_{\free}\otimes\mathbb{Q}\cong
H^{d}_{\mathbb{Q}}(\mu)=0 \) by the Borel-Weil-Bott theorem.
\section{The case \( n>m \)}\label{section:n>m}
\paragraph{\ref{section:n>m}.1}
If $s_1\leq m-1<n-1$, then a monomial $(b_0,b_1,\cdots,b_d)$ in
$E_{\nu}$ satisfies $b_0\geq 1$ since $b_1+\cdots+b_d\leq s_1<n$ and
$b_0+b_1+\cdots+b_d=n$. Set
\begin{multline*}
A=\{b=(b_0,b_1,\cdots,b_d)\in
\mathbb{N}^{d+1}\mid b_0+b_1+\cdots+b_d=n,\quad 1\leq b_0\leq
m+n-1-s_1,\\\quad b_i\leq s_i-s_{i+1}\quad \text{if} \quad 1\leq i\leq
d-1,\quad b_d\leq s_d \}.
\end{multline*}
Then we take as a basis for \( E_{\nu} \) the set
 $\{v_b\mid  b\in A\}$,  where $v_b$ is the
 monomial determined by $b$. Since \(
 m+n-1-s_{1}>n-1 \), we have
\begin{multline*}
A=\{b=(b_0,b_1,\cdots,b_d)\in
\mathbb{N}^{d+1}\mid b_0+b_1+\cdots+b_d=n,\quad 1\leq b_0\leq
n,\\\quad b_i\leq s_i-s_{i+1}\quad \text{if} \quad 1\leq i\leq
d-1, \quad b_d\leq s_d \}.
\end{multline*}
 For
 $0\leq i\leq d$, set
 $e_i=(0,0,\cdots,1,\cdots,0)\in\mathbb{Z}^{d+1}$ where $1$ is at the 
 $i$-th position.  
Then we take as a basis for \( F_{\nu} \) the set  $\{w_{b}\mid  b\in
A\} $,  where $w_b$ is the monomial determined by
 $b-e_0$. We set \( w_{b}=0 \) whenever \( b\notin
 A \). With these notations, we have for all \( u\in A \):
$$
f(v_u)=w_{u}+w_{u+e_0-e_1}+\cdots+w_{u+e_0-e_d}.
$$ 

We equip the set $A\subset\mathbb{N}^{d+1}$
with the reverse lexicographic  order. Then $u+e_0-e_i<u$ for all
$1\leq i\leq d$, and hence the matrix of $f$ with respect to basis $v_b$
and $w_b$ is lower triangular, and its  entries on the diagonal
are all \( 1 \). Hence the cokernel $H^d(\mu)_{\nu}$ of \( f_{\nu} \)
is zero if \( s_{1}<m \).

This proves that every weight of \( H^{d}(\mu) \) is \( \leq (m+n-1)\omega_{1}-m\alpha_{1} \).

\medskip
\paragraph{\ref{section:n>m}.2}
If \( s_{1}\geq m \), 
set $s_1=m+k$ with $k\geq0$. Let
\begin{align*}
 A=&\{b=(b_0,b_1,\cdots,b_d)\in \mathbb{N}^{d+1}\mid b_0+b_1+\cdots+b_d=n,\quad 1\leq b_0\leq n-1-k,\\
&b_i\leq s_i-s_{i+1}\quad \text{if} \quad 1\leq i\leq d-1, \quad b_d\leq s_d \},\\ 
C=&\{b=(0,b_1,\cdots,b_d)\in \mathbb{N}^{d+1}\mid b_1+\cdots+b_d=n,\\
& b_i\leq s_i-s_{i+1}\quad \text{if} \quad 1\leq i\leq d-1, \quad
                                                                        b_d\leq s_d \},\\ D=&\{b=(n-k-1,b_1,\cdots,b_d)\in \mathbb{N}^{d+1}\mid b_1+\cdots+b_d=k, \quad b_i\leq s_i-s_{i+1}\\
& \text{if} \quad 1\leq i\leq d-1,\quad b_d\leq s_d \}.
\end{align*}

We take the set $\{v_b\mid b\in A\cup C\}$ as the basis of $E_{\nu}$ where
$v_b$ is the monomial determined by $b=(b_0,\cdots,b_d)$. We take the
set $\{w_b\mid b\in A\}\cup \{u_b\mid b\in D\}$ as the basis of $F_{\nu}$
where $w_b$ is the monomial determined by $b-e_0$ and $u_b$ is the monomial determined by $b\in D$. 
\begin{convention}
  Let $b\in \mathbb{Z}^{d+1}$. If $b\notin A\cup C$,  we set
  $v_b=0$. If $b\notin A$, we set $w_b=0$. If $b\notin D$, we set $u_b=0$.
\end{convention}

With these notations, we have
\begin{equation}
  \begin{aligned}
    f(v_b)&=w_b+w_{b+e_0-e_1}+w_{b+e_0-e_2}+\cdots w_{b+e_0-e_d}\quad \text{if} \quad b_0\leq n-k-2\\
    f(v_b)&=w_b+u_{b-e_1}+u_{b-e_2}+\cdots+u_{b-e_d} \quad
    \text{if}\quad b_0=n-k-1.
  \end{aligned}
\end{equation}

Now we make a change of basis of  $E_{\nu}$ by defining
$v'_b$ for all $b\in A\cup C$ by:

\begin{itemize}
\item if $b_0=n-k-1$, we set $v'_b=v_b$;

 \item  if $b_0=n-k-2$, we set
  \begin{equation}
    v'_b=v_b-v_{b+e_0-e_1}-v_{b+e_0-e_2}-\cdots-v_{b+e_0-e_d} .
  \end{equation}
\end{itemize}

Let \( j\geq 1 \).  If we have already defined  $v'_b$ for all $b$
such that
$j\leq b_0\leq n-k-2$, we set
\begin{equation}
  v'_b=v_b-v'_{b+e_0-e_1}-v'_{b+e_0-e_2}-\cdots-v'_{b+e_0-e_d} \quad \text{if}\quad b_0=j-1.
\end{equation}
Hence $v'_b$ is defined for all $b\in A\cup C$.

Therefore, if $b_0=n-k-1$, we have:
\begin{equation}
\begin{aligned}
    f(v'_b)&=w_b+u_{b-e_1}+u_{b-e_2}+\cdots+u_{b-e_d}\\
    &=w_{b}+\binom{1}{1,0,\cdots,0}u_{b-e_{1}}+\binom{1}{0,1,0,\cdots,0}u_{b-e_{2}}+\cdots+\binom{1}{0,\cdots,0,1}u_{b-e_{d}}\\
    &=w_{b}+\sum_{b'\in D}\binom{1}{b_{1}-b_{1}',b_{2}-b_{2}',\cdots,b_{d}-b_{d}'}u_{b'}.
  \end{aligned}
\end{equation}
\begin{lemma}
  For all \( b\in A\cup C \), we have
  \begin{equation}
  \label{eq:5555*}
  f(v'_b)=w_b-(-1)^{n-k-b_{0}}\sum_{b'\in D}\binom{n-k-b_0}{b_1-b_1',b_2-b_2',\ldots,b_d-b_d'}u_{b'}.
\end{equation}
\end{lemma}
\begin{rmk}
  Since $b_0+\cdots+b_d=n$ if $b\in A\cup C$, we have
  $b_1+b_2+\cdots+b_d=n-b_0$. On the other hand, if $b'\in D$, then $b'_1+b'_2+\cdots+b'_d=n-1-b'_0=n-1-(n-1-k)=k$, hence $b_1-b_1'+b_2-b_2'+\cdots+b_d-b_d'=n-k-b_0$.
\end{rmk}

\begin{proof} We use  descending induction on \(b_{0}\)
  Clearly, \eqref{eq:5555*} is true if $b_0=n-k-1$.  If
  \eqref{eq:5555*} holds for all $b\in A\cup C$ such that
  $1\leq j\leq b_0\leq n-k-1$, then for all \( b\in A\cup C \) such
  that $b_0=j-1$, one has:
  \begin{align*}
    f(v'_b)=&f(v_b)-f(v'_{b+e_0-e_1})-\cdots-f(v'_{b+e_0-e_d})\\
    =&w_b+w_{b+e_0-e_1}+w_{b+e_0-e_2}+\cdots+w_{b+e_0-e_d}\\
            &-w_{b+e_0-e_1}+(-1)^{n-k-b_{0}-1}\sum_{b'\in D}\binom{n-k-b_0-1}{b_1-1-b_1', b_2-b_2',\cdots, b_d-b_d'}u_{b'}\\
            &-w_{b+e_0-e_2}+(-1)^{n-k-b_{0}-1}\sum_{b'\in D}\binom{n-k-b_0-1}{b_1-b_1', b_2-1-b_2',\cdots, b_d-b_d'}u_{b'}\\
            &-\cdots\\
            &-w_{b+e_0-e_d}+(-1)^{n-k-b_{0}-1}\sum_{b'\in
              D}\binom{n-k-b_0-1}{b_1-b_1', b_2-b_2',\cdots,
              b_d-1-b_d'}u_{b'}\\
  \end{align*}
  \begin{align*}
    =&w_b-(-1)^{n-k-b_{0}}\sum_{b'\in D}\sum_{i=1}^d\binom{n-k-b_0-1}{b_1-b_1',b_2-b_2',\ldots,b_i-1-b_i',\ldots,b_d-b_d'}u_{b'}\\
    =&w_b-(-1)^{n-k-b_{0}}\sum_{b'\in D}\binom{n-k-b_0}{b_1-b_1',b_2-b_2',\ldots,b_d-b_d'}u_{b'}.
  \end{align*}
  This proves the lemma.
\end{proof}
Therefore, the matrix of $f$ with respect to the bases $v'_b$ and $w_b,
u_b$ is of the form
\begin{equation*}
  \begin{array}{cc}
&  \begin{array}{ccclcccc} & &  & A  & & & &D \end{array} 
\\ 
\begin{array}{c}
\\
\\
A
\\
\\
\\
C
\end{array} 
 & \left(
\begin{array}{cccc|c}
1 & 0 & \cdots & 0 & *
\\
0 & 1 & \ddots & \vdots & * 
\\
\vdots &\ddots & \ddots & 0 & \vdots 
\\
0 & \cdots & 0 & 1 & * 
\\
\hline 
0& \cdots & 0 & 0 
&
M
\end{array}
\right)
\end{array} 
\end{equation*}
where the rows of $M$ are indexed by $C$, its columns  by $D$, and the
entry corresponding to \(  b\in C \) and \( b'\in D \)
is $(-1)^{n-k+1}\binom{n-k}{b_1-b_1',b_2-b_2',\ldots,b_d-b_d'}$.(One
has \( n-k-b_{0}=n-k \) since \( b_{0}=0
\) for \( b\in C \)).

We thus obtain the following proposition.
\begin{proposition} Let \( n>m\geq 0 \).
  \begin{enumerate}
  \item Every weight of \(
    H^{d}(m,0,\cdots,0,-n-d) \) is \( \leq (n-m-1,m,0,\cdots,0) \).
  \item For \( (k,s_{2},\cdots,s_{d}) \) such
    that
    \begin{displaymath}
      \nu=(n-m-1,m,0,\cdots,0)-k\alpha_{1}-s_{2}\alpha_{2}-\cdots-s_{d}\alpha_{d}
    \end{displaymath}
    is dominant, the \( \nu \)-weight space of
    \( H^{d}(m,0,\cdots,0,-n-d) \) is isomorphic as an abelian group
    to the cokernel of the matrix
    \begin{equation}
      \Bigg(\binom{n-k}{b_1-b_1',b_2-b_2',\ldots,b_d-b_d'}\Bigg)_{\substack{b\in
          C\\b'\in D}}\label{eq:36c4f9bf14becd66}
    \end{equation}
    where by setting \( s_{1}=m+k \), we have
    \begin{align*}
      C=&\{b=(0,b_1,\cdots,b_d)\in \mathbb{N}^{d+1}\mid b_1+\cdots+b_d=n,\\
        & b_i\leq s_i-s_{i+1}\quad \text{if} \quad 1\leq i\leq d-1, \quad
          b_d\leq s_d \},\\ D=&\{b=(n-k-1,b_1,\cdots,b_d)\in \mathbb{N}^{d+1}\mid b_1+\cdots+b_d=k, \quad b_i\leq s_i-s_{i+1}\\
        & \text{if} \quad 1\leq i\leq d-1, \quad b_d\leq s_d \}.     
    \end{align*}
  \end{enumerate}
\end{proposition}

\section{On the wall}
We now suppose that \( m=n \). Then we have proved that every dominant
weight of
 \( H^{d}(\mu) \) is of the form \( \nu=(2n-1)\omega_{1}-s_{1}\alpha_{1}-s_{2}\alpha_{2}-\cdots-s_{d}\alpha_{d} \)
with \( s_{1}\geq n \) and \( s_{1}\geq s_{2}\geq\cdots\geq s_{d}
\). Set \( h_{i}=s_{i}-s_{i+1} \) for \( 1\leq d-1 \) and \(
h_{d}=s_{d} \), then the fact that \( \nu \) is dominant  implies that
\( h_{1}\geq
h_{2}\geq \cdots\geq h_{d}\geq 0\). Set \( k=s_{1}-n\geq 0
\). Then as a \(
\mathbb{Z} \)-module,
the weight space \( H^{d}(\mu)_{\nu} \) is isomorphic to the cokernel
of the matrix \( M \) whose rows are indexed by
\begin{displaymath}
  C=\{b=(0,b_1,\cdots,b_d)\in \mathbb{N}^{d+1}\mid b_1+\cdots+b_d=n,\;
 b_i\leq h_{i} \}
\end{displaymath}
and whose columns are indexed by
\begin{displaymath}
 D=\{b=(n-k-1,b_1,\cdots,b_d)\in \mathbb{N}^{d+1}\mid b_1+\cdots+b_d=k,\; b_i\leq h_{i}\}
\end{displaymath}
and the entry corresponding to \( b\in C \) and \( b'\in D \)
is $\binom{n-k}{b_1-b_1',b_2-b_2',\ldots,b_d-b_d'}$.

This is a square matrix. In fact, there exists a bijection \(
\Phi:C\to D \) defined by \(
\Phi((0,b_{1},\cdots,b_{d}))=(n-k-1,h_{1}-b_{1},h_{2}-b_{2},\cdots,h_{d}-b_{d})
\) since \( h_{1}+h_{2}+\cdots+h_{d}=s_{1}=n+k \). The determinant of
this matrix has been calculated by Proctor (\cite{Proc90}
Cor.1). More
precisely, set \( h=(h_{1},\cdots,h_{d}) \) and for all \( \ell\geq 0
\), let
\begin{displaymath}
  C(d,h,\ell)=\{(b_{1},\cdots,b_{d})\mid b_{1}+\cdots+b_{d}=\ell,\; b_{i}\leq
  h_{i}\}.
\end{displaymath}
  For each \( \ell \), set \(
\delta_{\ell}=|C(d,h,\ell)|-|C(d,h,\ell-1)| \) (we use the convention
that  \( C(d,h,-1)=\emptyset \)) and \(
S_{\ell}=|C(d,h,0)|+|C(d,h,1)|+\cdots+|C(d,h,\ell)| \). Fix some
ordering of the elements of \( C(d,h,k) \). Since there is a bijection from \( C(d,h,k) \) to \( C(d,h,n) \) via \(
(b_{1},\cdots,b_{d})\mapsto (h_{1}-b_{1},\cdots,h_{d}-b_{d}) \), we
can order the elements of \( C(d,h,n) \) with the same ordering.
With these notations, one has the following 
\begin{proposition}[Proctor]\label{prop:proctor}
  If \( d\geq 1 \) and \( h_{1},\cdots,h_{d}\geq 1 \), then
    \begin{multline}
      \det
      \Bigg(\binom{n-k}{b_1-b_1',b_2-b_2',\ldots,b_d-b_d'}\Bigg)_{\substack{b\in
          C(d,h,n)\\b'\in C(d,h,k)}} \\
      =(-1)^{S_{k'}}\frac{\prod_{b'\in
          C(d,h,k)}b'_{1}!b'_{2}!\cdots b'_{d}!}{\prod_{b\in
          C(d,h,n)}b_{1}!b_{2}!\cdots b_{d}!}
      \prod_{\ell=0}^{k}[(\ell+1)(\ell+2)\cdots(\ell+n-k)]^{\delta_{k-\ell}},
\label{eq:9a0eb8726687ebf7}
\end{multline}
where \( k' \) is the largest odd integer \( \leq k \).
\end{proposition}
\begin{proof}
  Basically, this is just \cite{Proc90} Cor.1. The only thing we need
  to verify is that \( k<\frac{1}{2}(n+k) \) (corresponding to the
  hypothesis \( k<\frac{1}{2}R \) in the article of
  Proctor). But since \( \nu=(2n-1)\omega_{1}-s_{1}\alpha_{1}-s_{2}\alpha_{2}-\cdots-s_{d}\alpha_{d} \) is
dominant, we have  \( 0\leq 2n-1-2s_{1}+s_{2}=2n-1-s_{1}-h_{1}
\). Since \( h_{1}\geq 1 \), one has \(0 \leq
2n-1-s_{1}-1=n-2-k \), hence \( k\leq n-2 \), which implies \( k<\frac{1}{2}(n+k) \).
\end{proof}

In fact, the hypothesis \( h_{1},\cdots,h_{d}\geq 1 \) in the
proposition is not necessary. In our setting, we have \( h_{1}\geq
h_{2}\geq\cdots\geq h_{d}\geq 0 \). Let \( d_{0} \) be the largest integer
such that \( h_{d_{0}}\geq 1 \), then we have \( h_{1}\geq \cdots\geq
h_{d_{0}}\geq 1 \) and \( h_{d_{0}+1}=\cdots=h_{d}=0 \). Set \(
\bar{h}=(h_{1},\cdots,h_{d_{0}}) \), then \(
C(d,h,\ell)=C(d_{0},\bar{h},\ell)\times
\{(\underbrace{0,\cdots,0}_{d-d_{0} \text{ times}})\} \)  for all \( \ell
\) (intuitively, the set \( C(d,h,\ell) \) is just the set \(
C(d_{0},\bar{h},\ell) \), with some extra zeros added to each element
on the tail). Using  \autoref{prop:proctor} for \( d_{0} \) and \( \bar{h} \), we get
\begin{align*}
  &\det
      \Bigg(\binom{n-k}{b_1-b_1',b_2-b_2',\ldots,b_d-b_d'}\Bigg)_{\substack{b\in
          C(d,h,n)\\b'\in C(d,h,k)}} \\
      &=\det
      \Bigg(\binom{n-k}{b_1-b_1',b_2-b_2',\ldots,b_{d_{0}}-b_{d_{0}}',0,\cdots,0}\Bigg)_{\substack{b\in
        C(d,h,n)\\b'\in C(d,h,k)}} \\
  &=\det
      \Bigg(\binom{n-k}{b_1-b_1',b_2-b_2',\ldots,b_{d_{0}}-b_{d_{0}}'}\Bigg)_{\substack{b\in
    C(d_{0},\bar{h},n)\\b'\in C(d_{0},\bar{h},k)}} \\
\end{align*}
\begin{align*}
  &=(-1)^{S_{k'}}\frac{\prod_{b'\in
          C(d_{0},\bar{h},k)}b'_{1}!b'_{2}!\cdots b'_{d_{0}}!}{\prod_{b\in
          C(d_{0},\bar{h},n)}b_{1}!b_{2}!\cdots b_{d_{0}}!}
    \prod_{\ell=0}^{k}[(\ell+1)(\ell+2)\cdots(\ell+n-k)]^{\delta_{k-\ell}}\\
  &=(-1)^{S_{k'}}\frac{\prod_{b'\in
          C(d_{0},\bar{h},k)}b'_{1}!b'_{2}!\cdots b'_{d_{0}}!0!\cdots 0!}{\prod_{b\in
          C(d_{0},\bar{h},n)}b_{1}!b_{2}!\cdots b_{d_{0}}!0!\cdots 0!}
    \prod_{\ell=0}^{k}[(\ell+1)(\ell+2)\cdots(\ell+n-k)]^{\delta_{k-\ell}}\\
  &=(-1)^{S_{k'}}\frac{\prod_{b'\in
          C(d,h,k)}b'_{1}!b'_{2}!\cdots b'_{d}!}{\prod_{b\in
          C(d,h,n)}b_{1}!b_{2}!\cdots b_{d}!}
      \prod_{\ell=0}^{k}[(\ell+1)(\ell+2)\cdots(\ell+n-k)]^{\delta_{k-\ell}}.
\end{align*}
Therefore, we can get rid of the hypothesis \( h_{1},\cdots,h_{d}\geq 1
\). Moreover, by the definitions, we have \((b_{1},\cdots,b_{d})\in
C(d,h,n) \) if and only if \( (0,b_{1},\cdots,b_{d})\in C \), and 
 \( (b_{1},\cdots,b_{d})\in C(d,h,k) \) if and only if \((
 n-k-1,b_{1},\cdots,b_{d})\in D \). Hence the matrix
\eqref{eq:7e42ab2aaf16c45f} is the same as the one in
\eqref{eq:9a0eb8726687ebf7}. On the other hand, since \(
H^{d}(n,0,\cdots,0,-n-d) \) is a \( \mathbb{Z} \)-module of finite
type, all maximal weights are dominant. But a dominant weight is \(
\leq (-1,n,0,\cdots,0) \) if and only if it is \( \leq
(0,n-2,1,0,\cdots,0) \), we thus obtain the following corollary.
\begin{cor}\label{cor:onthewall}
  Let \( n\geq 0 \).
  \begin{enumerate}
  \item Every weight of \(
    H^{d}(n,0,\cdots,0,-n-d) \) is \( \leq (0,n-2,1,0, \cdots,0) \).
  \item For \( (k,s_{2},\cdots,s_{d}) \) such
    that
    \begin{displaymath}
      \nu=(-1,n,0,\cdots,0)-k\alpha_{1}-s_{2}\alpha_{2}-\cdots-s_{d}\alpha_{d}
    \end{displaymath}
    is dominant, the \( \nu \)-weight space of
    \( H^{d}(n,0,\cdots,0,-n-d) \) is isomorphic as an abelian group
    to the cokernel of a matrix with integer coefficients whose
    determinant has absolute value  
    \begin{equation}
      \label{eq:a91238c5d6d67c77}
      \frac{\prod_{b'\in
          C(d,h,k)}b'_{1}!b'_{2}!\cdots b'_{d}!}{\prod_{b\in
          C(d,h,n)}b_{1}!b_{2}!\cdots b_{d}!}
      \prod_{\ell=0}^{k}[(\ell+1)(\ell+2)\cdots(\ell+n-k)]^{\delta_{k-\ell}},
    \end{equation}
    where \( h=(h_{1},\cdots,h_{d})=(n+k-s_{2},s_{2}-s_{3},\cdots,s_{d-1}-s_{d},s_{d}) \).
    \end{enumerate}
\end{cor}

\begin{cor}\label{cor:nsmall}
  Let \( p \) be a prime number such that \( p>n \). Then \( H^{d}(n,0,\cdots,0,-n-d)\) is without \( p \)-torsion.
\end{cor}
\begin{proof}
  If \( n<p \), then for all \( b=(b_{1},\cdots,b_{d})\in C(d,h,k)\cup C(d,h,n) \) and \(
  i\in\{1,\cdots,d\} \), we have \( b_{i}\leq n<p \). For all
  \( 0\leq \ell\leq k \), we have \( \ell+n-k\leq n<p \). Hence 
  the determinant \eqref{eq:a91238c5d6d67c77} is
  non-zero modulo \( p \), and its cokernel has no \( p \)-torsion. 
\end{proof}
\begin{cor}\label{cor:Epp}
  Let \( K \) be an arbitrary field of characteristic \( p>0 \).
  \begin{enumerate}
  \item Then the dominant weights of
    \( H^{d}_{K}(p,0,\cdots,0,-p-d) \) are exactly those
    \( \leq \lambda_{0}=(0,p-2,1,0,\cdots,0) \), each of multiplicity
    \( 1 \).
  \item As a consequence, one has
    \( H^{d}_{K}(p,0,\cdots,0,-p-d)\cong L_{K}(\lambda_{0}) \), where
    \( L_{K}(\lambda_{0}) \) is the simple module of highest weight
    \( \lambda_{0} \).
  \end{enumerate}

\end{cor}
\begin{proof}
  Denote by \( \mu \) the weight \( (p,0,\cdots,0,-p-d) \).  Let \( k,s_{2},\cdots,s_{d}\in\mathbb{N} \)
  such that \(
  \nu=(-1,p,0,\cdots,0)-k\alpha_{1}-s_{2}\alpha_{2}-\cdots-s_{d}\alpha_{d}=(-1-2k+s_{2},p+k+s_{3}-2s_{2},\cdots)
  \) is dominant. Then \(
  s_{2}\geq 2k+1\geq 1 \). Hence we have \( \nu\leq
  (-1,p,0,\cdots,0)-\alpha_{2}=\lambda_{0} \). Thus, by \autoref{cor:onthewall}, every
  dominant weight of \( H^{d}_{K}(\mu)\cong H^{d}(\mu)\otimes K \) is \( \leq \lambda_{0} \). Moreover, for every such weight \(
  \nu \), let us adopt the notations in \autoref{cor:onthewall} with
  \( n=p \). Since \( s_{2}\geq 2k+1 \), we have \(
  h_{1}=n+k-s_{2}\leq n+k-(2k+1)=n-k-1\leq p-1 \), hence \( p-1\geq
  h_{1}\geq h_{2}\geq \cdots\geq h_{d} \). Therefore, for every \(
  \ell\in\mathbb{N} \) and every \( b=(b_{1},\cdots,b_{d})\in
  C(d,h,\ell) \), we have \( b_{i}\leq p-1 \) for all \( i \). This
  implies that  neither the numerator nor the
  denominator on the left part of \eqref{eq:a91238c5d6d67c77} involves
  a factor \( p \). In the
  right part of \eqref{eq:a91238c5d6d67c77}, every factor is \( <n=p \)
  except for the term with \( \ell=k \), and one has \( \delta_{0}=1
  \). Hence the \( p \)-adic valuation
  of \eqref{eq:a91238c5d6d67c77} is exactly \( 1 \), which implies
  that the weight \( \nu \) is of multiplicity \( 1 \) in \(
  H^{d}_{K}(\mu) \).

  \footnote{The author is grateful to one of the referees for this
    proof of assertion (2) and pointing out the work of Seitz
    mentioned in the following remark.}On the other hand, by \cite{Sup83},  the set of
weights of the simple module
\( L_{K}(\lambda_{0})\) consists of all dominant weights \( \leq
\lambda_{0} \). Since \( L_{k}(\lambda_{0}) \) is a simple factor of
\( H^{d}_{K}(p,0,\cdots,0,-p-d) \) whose weights are all of
multiplicity \( 1 \) with \(
\lambda_{0} \) as the highest weight, we conclude that \( 
H^{d}_{K}(p,0,\cdots,0,-p-d) \cong L_{K}(\lambda_{0}) \).
\end{proof}

\begin{rmk}
 The corollary above shows that $L_{K}(\lambda_0)$ has one-dimensional weight spaces. 
Note that Seitz has shown (\cite{Sei87}, Prop.~6.1) that if a simple  $\SL_{d+1}(K)$-module $L_{K}(\mu)$ has 
one-dimensional weight spaces, then either $\mu$ is a fundamental weight $\omega_i$ or a multiple of $\omega_1$ or $\omega_d$, 
or $\mu = a \omega_i + (p-1-a) \omega_{i+1}$ for some $i \in \{1,\dots, d-1\}$ and 
$a\in \{0,\dots, p-1\}$, and our result shows that indeed $L_{K}((p-2)\omega_2 + \omega_3)$ has one-dimensional weight spaces. 
\end{rmk}
\begin{cor}\footnote{The author thanks one of the referees for
    suggesting this result.}
   Let \( K \) be an arbitrary field of characteristic \( p>0 \).
  Then the dominant weights of \( L_{K}(p-2,1,0,\cdots,0) \) are
  exactly those \( \leq (p-2,1,0,\cdots,0)
  \), each of multiplicity \( 1 \).
\end{cor}
\begin{proof}
  Consider the Levi factor \( G'\cong \GL_{d} \) corresponding to the
  parabolic \( P_{1} \), and let \( L'_{K}(\lambda_{0}) \) be the
  simple \( G' \)-module with highest weight \( \lambda_{0} \). 
  By \cite{Jan03} II 2.11 b), we know that \( L'_{K}(\lambda_{0}) \)
  is a \( T \)-submodule of \( L_{K}(\lambda_{0}) \), hence each weight is of
  multiplicity \( 1 \) by \autoref{cor:Epp}. By replacing \( d \) with
  \( d+1 \), we deduce that all weights of the simple module \( L_{K}(p-2,1,0,\cdots,0)
  \) are of multiplicity \( 1 \). On the other hand, by
  \cite{Sup83}, it contains all dominant weights \( \leq (p-2,1,0,\cdots,0)
  \), which concludes the proof. 
\end{proof}
\begin{cor}  \label{coro:375ffdf8785a7373}
  Let \( K \) be a field of characteristic \( p>0 \) and \( \mu_{n}=(n,0,\cdots,0,-n-d) \). Suppose that \( n=p+r \)
with either
 (i) \(  0\leq r\leq p-2 \) or (ii) \( r=p-1\) and \(d\geq 3 \).

Let  \( \lambda_{r}=r\omega_{1}+(p-r-2)\omega_{2}+(r+1)\omega_{3}=(-1,n,0,\dots,0)-(r+1)\alpha_{2} \)
in case (i) and \( \lambda_{r}=(-1,n,0,\dots,0)-p\alpha_{2}-\alpha_{3}
\) in case (ii).

Then \( H^{d}_{K}(\mu_{n}) \) contains the weight \( \lambda_{r} \), with
  multiplicity \( 1 \).
\end{cor}
\begin{proof}
 Let us adopt the notations in \autoref{cor:onthewall}. In case (i), the weight \(
 \lambda_{r} \) corresponds to 
 \begin{displaymath}
 (k,s_{2},\cdots,s_{d})=(0,r+1,0,\cdots,0).
\end{displaymath}
In case (ii), it
corresponds to
\begin{displaymath}
 (k,s_{2},\cdots,s_{d})=(0,r+1,1,\cdots,0).
\end{displaymath}
In both cases, we have \(
 h_{1}=n+k-s_{2}=p+r-r-1=p-1 \) and \(
 h_{1}\geq h_{2}\geq \dots\geq h_{d}\). Therefore, for every \(
  \ell\in\mathbb{N} \) and every \( b=(b_{1},\cdots,b_{d})\in
  C(d,h,\ell) \), we have \( b_{i}\leq p-1 \) for all \( i \). This
  implies that in both cases,  neither the numerator nor the
  denominator on the left part of \eqref{eq:a91238c5d6d67c77} involves
  a factor \( p \). Moreover, since in both cases, we have \( k=0 \) and \( \delta_{0}=1 \), the right part of
  \eqref{eq:a91238c5d6d67c77} equals to \( n! \), whose \( p \)-adic valuation is \( 1 \). Hence the
  \( p \)-adic valuation of \eqref{eq:a91238c5d6d67c77} is exactly \( 1 \), which implies that \( H^{d}_{K}(\mu_{n}) \) contains the
  weight \( \lambda_{r} \) with multiplicity \( 1 \).
\end{proof}
\begin{rmk}
  \begin{enumerate}
  \item
    In a companion paper \cite{LP19} with P. Polo, we extend
    \autoref{cor:nsmall} to the case \( p>n \) and \( m \) arbitrary
    and we improve on \autoref{cor:Epp} and
    \autoref{coro:375ffdf8785a7373} by showing
    that\( H^{d}_{K}(p+r,0,...,0,-p-r-d) \) is the simple module
    \( L(\lambda_r) \).
   \item By \autoref{cor:Epp}, every weight of
     \(H^{d}_{K}(\mu_{p})  \) has multiplicity \( 1 \). This
     is no longer true for \( H^{d}_{K}(\mu_{p+r}) \) if \( r\geq 1 \). For
     example, set \( d=3 \), \( p=3 \) and \( r=1 \). Then  \( H^{d}_{K}(\mu_{p+r})=H^{3}_{K}(4,0,-7) \) has
     three dominant weights: \( (1,0,2),(1,1,0) \) and \( (0,0,1)
     \). The first two are both of multiplicity \( 1 \), but the last
     one appears with multiplicity \( 3 \).
  \end{enumerate}

\end{rmk}

In general, the number \( \delta_{k-\ell} \) in
\eqref{eq:a91238c5d6d67c77} is not easy to calculate. But if we
suppose that \( h_{1}\geq k \), 
we have the following proposition:
\begin{proposition}\label{prop:m1geqk}
  If \( h_{1}\geq k \), then for all \( \ell\in\{0,\cdots,k\} \), we have
  \begin{equation}
    \label{eq:84c1f2ee55d0c7bf}
    \delta_{k-\ell}=|\{(b_{1},\cdots,b_{d})\in C(d,h,k)\mid b_{1}=\ell \}|.
  \end{equation}

  Therefore, we have
  \begin{equation}
    \label{eq:b856a50c1d2a2abe}
   \det
      \Bigg(\binom{n-k}{b_1-b_1',b_2-b_2',\ldots,b_d-b_d'}\Bigg)_{\substack{b\in
          C\\b'\in D}}
      =(-1)^{S_{k'}}\frac{\prod_{b\in
          C(d,h,n)}\binom{n}{b_{1},b_{2},\cdots,b_{d}}}{\prod_{b'\in C(d,h,k)}\binom{n}{b'_{1}+n-k,b'_{2},\cdots,b'_{d}}}.
    \end{equation}

  Moreover, if \( d=2 \) or \( 3 \), we are always in this case.
\end{proposition}

\begin{proof}
  Let \( \ell \in \{0,\cdots,k\} \). Set
  \begin{align*}
    I&=C(d,h,k-\ell)=\{(b_{1},\cdots,b_{d})\mid \sum b_{i}=k-\ell, \quad b_{i}\leq
       h_{i}\}\\
    J&=C(d,h,k-\ell-1)=\{(b_{1},\cdots,b_{d})\mid \sum b_{i}=k-\ell-1, \quad b_{i}\leq
       h_{i}\}.
  \end{align*}
  Then by definition, we have \( \delta_{k-\ell}=|I|-|J| \). Since \(
  h_{1}\geq k \), we have
  \begin{align*}
    I&=\{(b_{1},\cdots,b_{d})\mid \sum b_{i}=k-\ell, \quad b_{i}\leq
       h_{i}\text{ for }2\leq i\leq d,\quad b_{1}\leq k\}\\
    J&=\{(b_{1},\cdots,b_{d})\mid \sum b_{i}=k-\ell-1, \quad b_{i}\leq
       h_{i}\text{ for }2\leq i\leq d,\quad b_{1}\leq k\}.
  \end{align*}

  Define \( I'=\{b\in I \mid  b_{1}\geq 1\}\subset I \). We can
  construct a bijection between \( I' \) and \( J \). More precisely,
  define
  \begin{displaymath}
    \phi:I'\rightarrow B,\quad (b_{1},\cdots,b_{d})\mapsto (b_{1}-1,b_{2},\cdots,b_{n}).
  \end{displaymath}
  This is clearly a well-defined injection. On the other hand,
   for all \( (b_{1},\cdots,b_{d})\in J \), we have \(
   b_{1}\leq k-\ell-1\leq k-1\leq h_{1}-1 \), thus \(
   (b_{1}+1,b_{2},\cdots,b_{d})\in I' \) and \( \phi
   (b_{1}+1,b_{2},\cdots,b_{d})=(b_{1},\cdots,b_{d}) \).  Hence \( \phi \) is a
   bijection.

   Now we have
   \begin{align*}
     \delta_{k-\ell}&=|I|-|J|=|I\backslash I'|\\
                    &=|\{b\in I\mid b_{1}=0\}|\\
     &=|\{(0,b_{2},\cdots,b_{d})\mid b_{2}+\cdots+b_{d}=k-\ell, \quad
       b_{i}\leq h_{i}\}|\\
     &=|\{(\ell,b_{2},\cdots,b_{d})\mid \ell+b_{2}+\cdots+b_{d}=k,\quad
       b_{i}\leq h_{i}\}|\\
                    &=|\{b\in C(d,h,k)\mid b_{1}=\ell \}|,
   \end{align*}
   where the last equality is due to the fact that \( \ell\leq k\leq
   h_{1} \). This proves \eqref{eq:84c1f2ee55d0c7bf}.

   With this expression of \( \delta_{k-\ell} \), we have
   \begingroup
   \allowdisplaybreaks
  \begin{align*}
     &\det
      \Bigg(\binom{n-k}{b_1-b_1',b_2-b_2',\ldots,b_d-b_d'}\Bigg)_{\substack{b\in
          C\\b'\in D}} \\
      =&(-1)^{S_{k'}}\frac{\prod_{b'\in
          C(d,h,k)}b'_{1}!b'_{2}!\cdots b'_{d}!}{\prod_{b\in
          C(d,h,n)}b_{1}!b_{2}!\cdots b_{d}!}
     \prod_{\ell=0}^{k}[(\ell+1)(\ell+2)\cdots(\ell+n-k)]^{\delta_{k-\ell}}\\
     =&(-1)^{S_{k'}}\frac{\prod_{b'\in
          C(d,h,k)}b'_{1}!b'_{2}!\cdots b'_{d}!}{\prod_{b\in
          C(d,h,n)}b_{1}!b_{2}!\cdots b_{d}!}
     \prod_{\ell=0}^{k}[(\ell+1)(\ell+2)\cdots(\ell+n-k)]^{\sharp\{b\in
        C(d,h,k)\mid b_{1}=\ell \}}\\
    =&(-1)^{S_{k'}}\frac{\prod_{b'\in
          C(d,h,k)}b'_{1}!b'_{2}!\cdots b'_{d}!}{\prod_{b\in
          C(d,h,n)}b_{1}!b_{2}!\cdots b_{d}!}
     \prod_{\ell=0}^{k}\prod_{\substack{b\in
        C(d,h,k)\\\text{such that }
    b_{1}=\ell}}[(\ell+1)(\ell+2)\cdots(\ell+n-k)]
  \end{align*}
  (here the second product simply means taking \(
  (\ell+1)(\ell+2)\cdots(\ell+n-k) \) to the \( \sharp\{b\in
        C(d,h,k)\mid b_{1}=\ell \}\)-th power)
        \begin{align*}
     =&(-1)^{S_{k'}}\frac{\prod_{b'\in
          C(d,h,k)}b'_{1}!b'_{2}!\cdots b'_{d}!}{\prod_{b\in
          C(d,h,n)}b_{1}!b_{2}!\cdots b_{d}!}
     \prod_{\ell=0}^{k}\prod_{\substack{b\in
        C(d,h,k)\\\text{such that }
    b_{1}=\ell}}[(b_{1}+1)(b_{1}+2)\cdots(b_{1}+n-k)]  \\   
     =&(-1)^{S_{k'}}\frac{\prod_{b'\in
          C(d,h,k)}b'_{1}!b'_{2}!\cdots b'_{d}!}{\prod_{b\in
        C(d,h,n)}b_{1}!b_{2}!\cdots b_{d}!}
        \prod_{b\in
        C(d,h,k)}[(b_{1}+1)(b_{1}+2)\cdots(b_{1}+n-k)]\\
    =&(-1)^{S_{k'}}\frac{\prod_{b'\in
          C(d,h,k)}b'_{1}!b'_{2}!\cdots b'_{d}!}{\prod_{b\in
        C(d,h,n)}b_{1}!b_{2}!\cdots b_{d}!}
        \prod_{b'\in
        C(d,h,k)}[(b'_{1}+1)(b'_{1}+2)\cdots(b'_{1}+n-k)]\\ 
     =&(-1)^{S_{k'}}\frac{\prod_{b'\in
          C(d,h,k)}(b'_{1}+n-k)!b'_{2}!\cdots b'_{d}!}{\prod_{b\in
        C(d,h,n)}b_{1}!b_{2}!\cdots b_{d}!}\\
     =&(-1)^{S_{k'}}\frac{\prod_{b\in
          C(d,h,n)}\binom{n}{b_{1},b_{2},\cdots,b_{d}}}{\prod_{b'\in C(d,h,k)}\binom{n}{b'_{1}+n-k,b'_{2},\cdots,b'_{d}}}.
   \end{align*}
   \endgroup
   This proves \eqref{eq:b856a50c1d2a2abe}

Finally, if \( d=2 \), we have \(
\nu=(2n-1)\omega_{1}-s_{1}\alpha_{1}-s_{2}\alpha_{2}=(2n-1+s_{2}-2s_{1},s_{1}-2s_{2})
\). Since \( \nu \) is dominant, we have \( 0\leq
s_{1}-2s_{2}=2h_{1}-s_{1}=2h_{1}-n-k \), hence \( h_{1}\geq
\frac{1}{2}(n+k)\geq k \) since \( k\leq n \) by the proof of
\autoref{prop:proctor}.

If \( d=3 \), we have \(
\nu=(2n-1)\omega_{1}-s_{1}\alpha_{1}-s_{2}\alpha_{2}-s_{3}\alpha_{3}=(2n-1+s_{2}-2s_{1},s_{1}+s_{3}-2s_{2},
s_{2}-2s_{3}) \). Since \( \nu \) is dominant, we have
\begin{displaymath}
  2s_{2}\leq s_{1}+s_{3}\leq \frac{1}{2}(2n-1+s_{2})+\frac{1}{2}s_{2}=s_{2}+n-\frac{1}{2}.
\end{displaymath}
Hence \( s_{2}<n \), and \( h_{1}=s_{1}-s_{2}=n+k-s_{2}>k \). This  finishes the proof of
\autoref{prop:m1geqk}.  
\end{proof}

\begin{rmk}
  In fact, if \( h_{i}\geq k \) for an \( i\in\{1,\cdots,d\} \) (which
  implies \( h_{1}\geq k \)), then we have
  \begin{equation}
    \label{eq:578de7099ee509cd}
    \det
      \Bigg(\binom{n-k}{b_1-b_1',b_2-b_2',\ldots,b_d-b_d'}\Bigg)_{\substack{b\in
          C\\b'\in D}}
      =(-1)^{S_{k'}}\frac{\prod_{b\in
          C}\binom{n}{b_{1},b_{2},\cdots,b_{d}}}{\prod_{b'\in D}\binom{n}{b'_{1},\cdots,b'_{i-1},b'_{i}+n-k,b'_{i+1}\cdots,b'_{d}}}.
    \end{equation}
The proof is similar to the case \( i=1 \).
\end{rmk}

\section{The case \( G=\SL_{3} \) }\label{section:gl3}
Assume that \( d=2 \), i.e. \( G=\SL_{3} \). Let \( \alpha,\beta \) be
the simple roots, and \( \gamma=\alpha+\beta \).
\paragraph{\ref{section:gl3}.1}
The sets \( C \) and \( D \) are a
lot simpler. In this case, the multinomial coefficients are replaced
by binomial coeffeicients, and we have the following corollaries. 
\begin{cor}\label{prop:sl3above}
Let \( m\geq n>0 \).
  \begin{enumerate}
\item Every weight of \(
    H^{2}(m,-n-2) \) is \( \leq (m-n-1,n) \).
  \item For \( (t,k) \) such that
    \( \nu_{t,k}=(m-n-1,n)-k\alpha-t\beta \) is dominant, the \(
    \nu_{t,k} \)-weight space of \( H^{2}(m,-n-2) \) is isomorphic as
    an abelian group to the cokernel of the matrix
    \begin{equation}
      \label{eq:matrice1}
      D_{m,n,t,k}=\left(
        \begin{array}{cccc}
          \binom{m-k}{t-k}&\binom{m-k}{t-k-1}&\cdots&\binom{m-k}{t-2k+m-n}\\[6pt] 
          \binom{m-k}{t-k+1}&\binom{m-k}{t-k}&\cdots&\binom{m-k}{t-2k+m-n+1}\\[6pt] 
          \vdots&\vdots&\ddots&\vdots\\[6pt] 
          \binom{m-k}{t}&\binom{m-k}{t-1}&\cdots&\binom{m-k}{t-k+m-n}
        \end{array}
      \right)
    \end{equation}
    if \( m-n\leq k\leq t \), and is zero otherwise.
  \end{enumerate}

\end{cor}

\begin{cor} Let \( n>m\geq 0 \)
  \begin{enumerate}
\item Every weight of \(
    H^{2}(m,-n-2) \) is \( \leq (n-m-1,m) \).
\item  For \( (t,k) \) such that \(
  \nu_{t,k}=(n-m-1,m)-k\alpha-t\beta \) is dominant, the \(
    \nu_{t,k} \)-weight space of \( H^{2}(m,-n-2) \) is isomorphic as
    an abelian group to the cokernel of the matrix
\begin{equation}
    \label{eq:matrice1souslemur}
    D_{m,n,t,k}=\left(
    \begin{array}{cccc}
     \binom{n-k}{t-k+n-m}&\binom{n-k}{t-k+n-m-1}&\cdots&\binom{n-k}{t-2k+n-m}\\[6pt] 
    \binom{n-k}{t-k+n-m+1}&\binom{n-k}{t-k+n-m}&\cdots&\binom{n-k}{t-2k+n-m+1}\\[6pt] 
    \vdots&\vdots&\ddots&\vdots\\[6pt] 
    \binom{n-k}{t}&\binom{n-k}{t-1}&\cdots&\binom{n-k}{t-k}
    \end{array}
  \right)
\end{equation}
 if \( k\geq n-m \), and is isomorphic to \( \mathbb{Z}^{\min(t,k)-\max(0,t-m)+1}
 \) otherwise.
 \end{enumerate}

\end{cor}

\begin{rmk}\label{rmk:sl3mur}
  If \(\mu=(m,-n-2)\) is on the wall, i.e. \(m=n\), then the matrix
  \(D_{m,n,t,k}=D_{n,t,k}\) is square.
 More precisely, we have
   \begin{equation}
    \label{eq:matricemur}
    D_{n,t,k}=\left(
      \begin{array}{cccc}
        \binom{n-k}{t-k}&\binom{n-k}{t-k-1}&\cdots&\binom{n-k}{t-2k}\\[6pt] 
        \binom{n-k}{t-k+1}&\binom{n-k}{t-k}&\cdots&\binom{n-k}{t-2k+1}\\[6pt] 
        \vdots&\vdots&\ddots&\vdots\\[6pt] 
        \binom{n-k}{t}&\binom{n-k}{t-1}&\cdots&\binom{n-k}{t-k}
      \end{array}
    \right).
  \end{equation}

 While we can still apply the result of \cite{Proc90} Cor.1,  this determinant has also been calculated in
 \cite{Krat99} (2.17), which gives:
  \begin{equation}
    \label{eq:detmur}
    d_{n,t,k}=\det(D_{n,t,k})=\prod_{i=1}^{k+1}\prod_{j=1}^{t-k}\prod_{l=1}^{n-t}\frac{i+j+l-1}{i+j+l-2} 
    =\prod_{i=1}^{k+1}\frac{\binom{n-k+i-1}{t-k}}{\binom{t-k+i-1}{t-k}}
    =\prod_{i=0}^{k}\frac{\binom{n}{t-i}}{\binom{n}{i}}.
  \end{equation}
  
\end{rmk}
\paragraph{\ref{section:gl3}.2}
In the following, we fix an arbitrary field \( k \) of
characteristic \( p>0 \) and we use \( G \), \( B \), etc., to denote the
corresponding group scheme over \( k \) obtained by base change \(
\mathbb{Z}\to k \). Now we have \( H^{2}(m,-n-2)\cong
H^{1}(-m-2,n)^{*} \) and we can apply the results in \cite{Jan03} II.5.15.

For \( \lambda \) dominant, denote by \( L(\lambda) \) (resp. \( V(\lambda) \))  the simple \( G
\)-module (resp. Weyl mo\-du\-le) of highest weight \( \lambda \). If \(
\lambda \) is not dominant, we use the convention that \( L(\lambda)=V(\lambda)=0 \). Then
we have the following proposition.

In \cite{Liu19} Thm.1, the author has proved that if \( n=ap^{d}+r \)
with \( 1\leq a\leq p-1 \), \( d\geq 1 \) and \( 0\leq r<p^{d} \),
there exists an exact sequence
\begin{displaymath}
  0\to L(0,a)^{(d)}\otimes H^{2}(r,-r-2)\to H^{2}(n,-n-2)\to
  Q(n,-n-2)\to 0
\end{displaymath}
where \( Q(n,-n-2) \) is a certain quotient of \( V(n,-n-2) \). If \(
r<p \), we have \( H^{2}(r,-r-2)=0 \) according to
\autoref{cor:nsmall}, and hence \( H^{2}(n,-n-2)=Q(n,-n-2) \). We will
determine \( Q(n,-n-2) \) in this case. 
\begin{proposition}
  \label{thm:main3}
  If \(n=ap^{d}+r\) with \(a\in\{1,2,\cdots,p-1\}\) and \( r\in\{0,1,\cdots,p-1\} \),
  then we have an exact sequence of \( G \)-modules
  \begin{equation}\label{eq:resultat3}
    \xymatrix{
    0\ar[r]&L(p^{d}-1,(a-2)p^{d}+r)\ar[r]&V(r,n-2r-2)\ar[r]&H^{2}(n,-n-2)\ar[r]&0.
    }
  \end{equation}
\end{proposition}

\begin{rmk}
   If \( n=p^{2}-1 \), then \( H^{2}(n,-n-2)\cong H^{1}(-n-2,n)^{*}=0
   \) by \cite{Jan03} II.5.15 a) and \(
   V(r,n-2r-2)=V(p-1,(p-3)p+p-1)\cong L(p-1,(p-3)p+p-1) \) by
   \cite{Jan03} II 3.19 and Steinberg's tensor identity. Hence the
   proposition is true in this case and we may assume that \( n\neq
   p^{2}-1 \) in the proof.

   If \( a=1 \), then we have \( H^{2}(n,-n-2)\cong
   V(r,p^{d}-r-2)=V(r,n-2r-2) \) by \cite{Liu19} Thm.2. On the other
   hand, we have \( L(p^{d}-1,(a-2)p^{d}+r)=L(p^{d}-1,r-p^{d})=0 \)
   by our convention. Hence we can also suppose that
    \( a\geq 2 \) in the proof.
 \end{rmk}

 \begin{proof}
   By Serre duality, we have \( H^{2}(n,-n-2)\cong H^{1}(-n-2,n)^{*}
   \). According to \cite{Jan03} II.5.15, the socle of \(
   H^{1}(-n-2,n) \) is simple and isomorphic to \( L(n-2r-2,r)
   \). Since \( r<p \), \( (n-2r-2,r) \) is also the highest weight of
   \( H^{1}(-n-2,n) \) by the same proposition.
 Hence by duality,  \( H^{2}(n,-n-2) \) is generated by its highest
 weight \( (r,n-2r-2)\).  We thus
  have an exact sequence of \( G \)-modules
  \begin{equation}
    \begin{tikzcd}
      0\ar[r]&K\ar[r]&V(r,n-2r-2)\ar[r]&H^{2}(n,-n-2)\ar[r]&0.
    \end{tikzcd}\label{eq:9c8020501d7af93e}
  \end{equation}
It suffices to prove that \( K\cong L(p^{d}-1,(a-2)p^{d}+r) \)

1) First suppose that \( r=0 \). In this case, \( n=ap^{d} \) and the
Weyl module \( V(0,ap^{d}-2) \) has no multiplicity. The submodule
structure of \( V(0,ap^{d}-2) \) has been determined by Doty  (\cite{Doty85}).

  As in \autoref{prop:sl3above}, set
  \begin{displaymath}
    \nu_{t,k}=(2n-1)\omega_{1}-(n+k)\alpha-t\beta=(t-2k-1,n+k-2k).
  \end{displaymath}
  We want to prove that
  \begin{displaymath}
    K\cong L_0 = L\big(p^d - 1, (a-2) p^d\big) = L(p^d - 1, n-2 p^d) = L(\nu_{p^d,0}).
  \end{displaymath}

  Using the same notation as in \autoref{rmk:sl3mur}, we have
  \begin{displaymath}
    \det(D_{n,p^{d},0})=\binom{n}{p^{d}}=\binom{ap^{d}}{p^{d}}\equiv
  \binom{a}{1}\not\equiv 0
  \pmod p.
\end{displaymath}
Hence the matrix reduced modulo \( p \) is invertible and hence its cokernel is
zero. This means that \( H^{2}(n,-n-2) \) does not contain the weight
\( \nu_{p^{d},0} \), thus the \( \nu_{p^{d},0} \)-weight space is
contained in \( K \).

To prove that \( K= L_{0} \), we will use the results in
\cite{Doty85}. Doty considers the module \( H^{0}(m,0) \), while we
consider its dual \( V(0,m) \), for
\begin{displaymath}
  m = ap^d - 2 = (p-2) + \sum_{u=1}^{d-1} (p-1) p^u + (a-1) p^d.
\end{displaymath}

As in \cite[2.3]{Doty85}, for \(u = 0,\dots, d\), denote by
\(c_u(m)\) the \(u\)-th digit of the \(p\)-adic expansion of \(m\);
we thus have \(c_0(m) = p-2\),  \(c_u(m) = p-1\) for \(u = 1,\dots, d-1\),
 \(c_d(m) = a-1\) and 
 \(c_u(m) = 0\) for \(u > d\). 

\medskip As in \cite{Doty85}, Prop. 2.4, denote by \(E(m)\)  
the set of all  \(d\)-tuples \((a_1,\dots, a_d)\) of integers in 
\(\{0,1,2\}\) satisfying
\begin{equation}
0\leq c_{u}(m)+a_{u+1}p-a_{u}\leq 3(p-1)\label{eq:edffa97d563ec438}
\end{equation}
for all \( u=0,\cdots,d \) (here we use the convention that \(
a_{0}=a_{d+1}=0 \)).
\begin{lemma}
  We have \( E(m)=\{0,1\}^{d} \).
\end{lemma}
\begin{proof} We prove by induction on \(u_{0}\in\{1,\cdots,d\} \) that
  \eqref{eq:edffa97d563ec438} holds for \(u= 0,\cdots,u_{0}-1 \) if
  and only if \( 0\leq a_{u}\leq 1 \) for all \( 1\leq u\leq u_{0} \). For \( u=0 \) in
  \eqref{eq:edffa97d563ec438}, we get
  \begin{equation}
    0 \leq c_0(m) + a_1 p = p - 2 + a_1 p \leq 3(p-1).
  \end{equation} 
  This inequality holds if and only if \( 0\leq a_{1}\leq 1 \).

  Suppose that for some \( 1\leq u_{0}\leq d-1 \), we have proved that
  \eqref{eq:edffa97d563ec438} holds for \( u=0,\cdots,u_{0}-1 \) if
  and only if \( 0\leq a_{u}\leq 1 \) for all \( 1\leq u\leq u_{0}
  \). Now by taking \( u=u_{0} \), \eqref{eq:edffa97d563ec438} gives
  \begin{equation}\label{eq:u0induction}
    0 \leq c_{u_{0}}(m) + a_{u_{0} +1} p - a_{u_{0}} = p - 1 + a_{u_{0}+1} p - a_{u_{0}} \leq 3(p-1).
  \end{equation}
   Assuming \( 0\leq a_{u_{0}}\leq 1 \), \eqref{eq:u0induction}
  holds if and only if \( 0\leq a_{u_{0}+1}\leq 1 \). Hence by
  induction, \eqref{eq:edffa97d563ec438} holds for \( u=0,\cdots,d-1 \) if
  and only if \( 0\leq a_{u}\leq 1 \) for all \( 1\leq u\leq d
  \).   At last, for \(u = d\), \eqref{eq:edffa97d563ec438} becomes
  \[
    0 \leq c_d(m) - a_d = a-1 - a_d \leq 3(p-1)
  \]
  which is automatically satisfied if \( 0\leq a_{d}\leq 1 \) since \(2\leq
  a\leq p-1\).
 This finishes the proof of the lemma.
\end{proof}

\smallskip Since \(V(0,m)\) is the dual of \(H^0(m,0)\),
it contains a simple module \(L(x,y)^* = L(y,x)\) 
if and only if  \(L(x,y)\) is a quotient of \(H^0(m,0)\). By
\cite{Doty85} Thm.2.3, the submodule lattice of \( H^{0}(m,0) \) is
equivalent with the lattice of \( E(m) \) equipped with the partial
order \( (a_{1},\cdots,a_{d})\leq (a'_{1},\cdots,a'_{d}) \) if and
only if \( a_{i}\leq a'_{i} \) for all \( i \). As in \cite{Doty85} 2.4, for \(
a=(a_{1},\cdots,a_{d})\in E(m)  \) and \( u\in\{0,\cdots,d\} \),  let \(N_u(a)\)
(resp. \(R_u(a)\)) be the 
quotient
(resp. the remainder) of the  Euclidiean division of 
\(c_u(m) + a_{u+1} p - a_u\) by \(p-1\). (And one takes \(a_0 = 0 =
a_{d+1}\)). Then the simple factor of \( H^{0}(m,0) \) corresponding to
\( a\in E(m) \) is
\( L(b_{1}-b_{2},b_{2}-b_{3}) \), where \( b_{j} \) is determined by
the following rule
(cf. \cite{Doty85} top of the page 379):
\begin{equation}
  \label{eq:1dff54036e966e97}
  c_{u}(b_{j})=
  \begin{cases}
    p-1,\quad &\text{if}\; j\leq N_{u}(a),\\
    R_{u}(a),\quad &\text{if}\; j=N_{u}(a)+1,\\
    0,\quad &\text{if}\; j>N_{u}(a)+1.
  \end{cases}
\end{equation}

Taking this into account,  we know that \(V(0,ap^d - 2)\) contains a
unique simple submodule \(L(\nu)^*\), which corresponds to the maximal
element \(e = (1,\dots,1)\) of \(E(m)\). Suppose \(
\nu=(b_{1}-b_{2},b_{2}-b_{3}) \). We will calculate \(
b_{1},b_{2},b_{3} \) using \eqref{eq:1dff54036e966e97}.
In this case, for \(u = 0,\dots, d\),  \(N_u(e)\)
(resp. \(R_u(e)\)) is the 
quotient
(resp. the remainder) of the  Euclidiean division of 
\(c_u(m) + e_{u+1} p - e_u\) by \(p-1\), where \( e_{0}=e_{d+1}=0 \)
and \( e_{1}=e_{2}=\cdots=e_{d}=1 \). 
We thus have: 
\[
c_0(m) + e_1 p - e_0 = p-2 + p - 0 = 2p-2, \qquad\text{thus \(N_0(e) = 2\) and } R_0(e) = 0.
\]
Then for \(u = 1,\dots, d-1\), we have: 
\[
c_u(m) + e_{u+1} p - e_u = p-1 + p - 1 = 2p-2, \qquad\text{thus \(N_u(e) = 2\) and } R_u(e) = 0.
\]
Finally, for \(u = d\) we have \(c_d(m) + e_{d+1} p - e_d = a-1 + 0 -
1 = a-2\) and hence \(N_d(e) = 0\) and \(R_d(e) = a-2\). Therefore, the coefficients 
\(c_u(b_j)\) of the \(p\)-adic expasion  of \(b_j\) are given for \(j = 1\) by: 
\[
c_u(b_1) = \begin{cases} 
p-1 & \text{ if \(j = 1 \leq N_u(e)\) i.e.  if } u = 0, 1, \dots, d-1
\\ 
R_d(e) = a-2 & \text{ for \(u= d\) since } j = 1 = N_d(e)+1.
\end{cases}
\]
Then, for \(j = 2\) the coefficients \(c_u(b_2)\) are given by 
\[
c_u(b_2) = \begin{cases} 
p-1 & \text{ if \(j = 2 \leq N_u(e)\) i.e. if } u = 0, 1, \dots, d-1
\\ 
0 & \text{ for \(u= d\) since } j = 2 > N_d(e)+1 . 
\end{cases}
\]
Finally, for  \(j = 3\) the coefficients \(c_u(b_3)\) are given by
\[
c_u(b_3) = \begin{cases} 
R_u(e) = 0 & \text{ si \(j = 3 = N_u(e) +1 \) i.e. if } u = 0, 1, \dots, d-1
\\ 
0 & \text{ pour \(u= d\) car } j = 3 > N_d(e)+1 . 
\end{cases}
\]
We thus obtain the triplet \((p^d-1 + (a-2) p^d, p^d - 1, 0)\) and
then the dominant weight
\[
\nu = (a-2)p^d \omega_1 + (p^d - 1) \omega_2
\]
and hence \(V(0, ap^d-2)\) contains as unique simple submodule the
simple module considered earlier:
\[
L_0 = L(\nu)^{*} = L(p^d - 1, (a-2)p^d) = L(p^d - 1, n - 2p^d) = L(\nu_{p^d,0}). 
\]
Since we have proved that the weight \(\nu_{p^{d},0}  \) is contained
in \( K \), we have \( L_{0}\subset K \). It remains to prove that \(
K\subset L_{0} \).

\medskip 
Still according to  \cite{Doty85}, Theorem 2.3 and \S 2.4,
the socle of \(V(0,m)/L_0\) is the direct sum of the simple 
modules  \(L(e^i)^*\), for \(i = 1,\dots, d\), where each \(e^i\)
means the \(d\)-tuple: 
\[
(1,\dots, 1,0,1,\dots, 1)
\]
with the unique  \(0\) at the \(i\)-th position. We need to determine
the highest weight of \(L(e^i)\),
still with the help of \eqref{eq:1dff54036e966e97}. 
This time we have,
\[
N_{i-1}(e^i) = 0, \qquad R_{i-1}(e^i) = p-2, \qquad N_{i}(e^i) = 2, \qquad R_{i}(e^i) = 1,
\]
if \(i\leq d-1\), and
\[
N_{d-1}(e^d) = 0, \qquad R_{d-1}(e^d) = p-2, \qquad N_{d}(e^d)=0,\qquad R_d(e^d)=a-1.
\]
Thus the highest weight \(\lambda_{i}\) of \(L(e^{i})^{*}\)
is \( \nu_{t_{i},k_{i}} \) with
\((t_{i},k_{i})=(p^d+p^{i-1},p^i)\) for \(i=1,\cdots,d-1\) and \(\lambda_d=\nu_{t_{d},k_{d}}\)
with \((t_{d},k_{d})=(p^{d-1},0)\).

If \(1\leq i\leq d-1\), with the notation of \autoref{rmk:sl3mur}, we have
\begin{align*}
 d_{n,t_{i},k_{i}}&=\frac{ \prod_{l=0}^{k_{i}}\binom{n}{t_{i}-l}}{\prod_{l=0}^{k_{i}}\binom{n}{l}}\\
 v_p(d_{n,t_{i},k_{i}})&=(j+1)d-\sum_{l=0}^{k_{i}}v_p(t_{i}-l)-jd+\sum_{l=1}^{k_{i}}v_p{l}\\
&=d-v(\binom{t_{i}}{k_{i}})-v(t_{i}-k_{i})=d-v(\binom{t_{i}}{k_{i}})-(i-1)\\
&\geq d-(d-i)-(i-1)=1,
\end{align*}
where the last inequality results from the \( p \)-adic expansion of 
\(k_{i}=p^i\) and \(t_{i}-k_{i}=p^d-p^i+p^{i-1}\).
This means that  \(p\) divides \(d_{n,t_{i},k_{i}}\) and the cokernel
of \( D_{n,t_{i},k_{i}} \) is non-trivial. Hence \( H^{2}(n,-n-2) \)
contains the weight \( \lambda_{i} \) and \( L(\lambda_{i}) \) does
not exist in \( K \).

For \(i=d\), where \((t_{i},k_{i})=(p^{d-1},0)\), we have
\(d_{n,t_{d},k_{d}}=\binom{n}{t_{d}}=\binom{ap^d}{p^{d-1}}\) ,  and
\(v_p(d_{n,t_{d},k_{d}})=1\). Hence \( L(\lambda_{d}) \) does not
exists in \( K \) either. This proves that \( K=L_{0} \), i.e. there
is an exact sequence of \( G \)-modules:
\begin{equation}
\xymatrix{
    0\ar[r]&L(p^{d}-1,(a-2)p^{d})\ar[r]&V(0,ap^{d}-2)\ar[r]&H^{2}(ap^{d},-ap^{d}-2)\ar[r]&0.
  }\label{eq:ser0}
\end{equation}

\medskip
2) Suppose now that \( 1\leq r\leq p-1 \).

Set \(\lambda=(0,ap^{d}-2)\) and \(\mu=(r,ap^{d}-r-2)\). Then the
facet containing
 \(\lambda\) is defined by
\[F=\{\nu\in X(T)\mid  0<\langle \nu+\rho, \alpha^{\vee}\rangle<p,
  ap^{d}-p<\langle \nu+\rho, \beta^{\vee}\rangle<ap^{d},
  \langle \nu+\rho, \gamma^{\vee}\rangle=ap^{d}\}
\]
and \(\mu\) belongs to the closure \( \bar{F} \) (in fact it belongs
to \( F \) if \( r\neq p-1 \)).

Let \(T_{\lambda}^{\mu}\) be the translation functor from \( \lambda
\) to \( \mu \), which is exact (cf. \cite{Jan03} II.7.6). Then we
have \( T_{\lambda}^{\mu}V(\lambda)=V(\mu) \) by \cite{Jan03}
II.7.11. Apply \(
T_{\lambda}^{\mu} \) to the exact sequence, one obtains:
\eqref{eq:ser0}:
\begin{equation}
\xymatrix{
    0\ar[r]&T_{\lambda}^{\mu}L(p^{d}-1,(a-2)p^{d})\ar[r]&V(\mu)\ar[r]&T_{\lambda}^{\mu}H^{2}(ap^{d},-ap^{d}-2)\ar[r]&0.
  }\label{eq:transr0}
\end{equation}

Define the elements \(w_{1}, w_{2}\in W_{p}\) by
\[w_{1}\cdot
\nu=\nu-(\langle \nu+\rho,\beta^{\vee}\rangle-(a-1)p^{d})\beta\] and
\[w_{2}\cdot
\nu=s_{\beta}\cdot(\nu-(\langle
\nu+\rho,\beta^{\vee}\rangle-ap^{d})\beta)=\nu-ap^{d}\beta.\]
Then \((p^{d}-1,(a-2)p^{d})=w_{1}\cdot \lambda\) belongs to the facet.
\begin{multline*}
  F'=\{\nu\in X(T)\mid  \langle \nu+\rho, \alpha^{\vee}\rangle=p^{d},
  (a-2)p^{d}<\langle \nu+\rho, \beta^{\vee}\rangle<(a-2)p^{d}+p,\\
  (a-1)p^{d}<\langle \nu+\rho, \gamma^{\vee}\rangle<(a-1)p^{d}+p\}.
\end{multline*}
Hence
\( w_{1}\cdot\mu=(p^{d}-1,(a-2)p^{d}+r) \) belongs to the upper
closure of \( F' \) (see \cite{Jan03} II.6.2(3) for the definition of
the upper closure \( \hat{F'} \) of \( F' \) ). Therefore,  by \cite{Jan03}
  II 7.15, we have
  \[T_{\lambda}^{\mu}L(p^{d}-1,(a-2)p^{d})=T_{\lambda}^{\mu}L(w_{1}\cdot\lambda)
    \cong L(w_{1}\cdot \mu)=L(p^{d}-2, (a-2)p^{d}+r).\]

  Similarly, \(w_{2}\cdot\lambda=(ap^{d},-ap^{d}-2)\), hence by
  \cite{Jan03} II.7.11, we have
  \[T_{\lambda}^{\mu}H^{2}(ap^{d},-ap^{d}-2)=T_{\lambda}^{\mu}H^{2}(w_{2}\cdot
    \lambda)
    \cong H^{2}(w_{2}\cdot\mu)=H^{2}(ap^{d}+r,-ap^{d}-r-2).
  \]

  Therefore, the exact sequence  (\ref{eq:transr0}) becomes
  \eqref{eq:resultat3}. This proves \autoref{thm:main3}.
\end{proof}

\def\refname{References}

\end{document}